\documentclass[12pt]{amsart}
\usepackage{mathrsfs}
\usepackage{amsfonts}
\usepackage{amssymb}
\usepackage{amsfonts, amscd, amsmath, mathrsfs, amssymb, amsthm, amsxtra, bbding, epsfig, eucal, graphicx, latexsym, url, mathbbol, bbold}
\usepackage[papersize={7.7in,10.3in},textwidth=14.8cm,textheight=20.5cm,centering]{geometry}
\usepackage{enumerate}

\usepackage{xcolor}
\definecolor{cite}{rgb}{0.00,0.60,1.00}
\definecolor{url}{rgb}{1.00,0.10,0.80}
\definecolor{link}{rgb}{0.00,0.00,1.00}
\usepackage[colorlinks,linkcolor=link,urlcolor=url,citecolor=cite,pagebackref,breaklinks]{hyperref}

\hypersetup{
pdfstartpage=1,
pdfstartview=FitH}

\DeclareFontFamily{U}{mathx}{\hyphenchar\font45}
\DeclareFontShape{U}{mathx}{m}{n}{
      <5> <6> <7> <8> <9> <10>
      <10.95> <12> <14.4> <17.28> <20.74> <24.88>
      mathx10
      }{}
\DeclareSymbolFont{mathx}{U}{mathx}{m}{n}
\DeclareMathAccent{\widecheck}{\mathalpha}{mathx}{"71}

 \usepackage{caption} 
\numberwithin{equation}{section}

\allowdisplaybreaks

\newtheorem{theorem}{Theorem}[section]
\newtheorem{lemma}{Lemma}[section]

\newtheorem{definition}{Definition}[section]
\newtheorem{proposition}{Proposition}[section]

\newtheorem{corollary}{Corollary}[section]

\makeatletter
\newcounter{roem}
\renewcommand{\theroem}{\Roman{roem}}

\newcommand{\c@org@eq}{}
\let\c@org@eq\c@equation
\newcommand{\org@theeq}{}
\let\org@theeq\theequation

\newcommand{\setroem}{
\let\c@equation\c@roem
 \let\theequation\theroem}

\newcommand{\setarab}{
\let\c@equation\c@org@eq
\let\theequation\org@theeq}
\makeatother

\newtheorem*{claim*}{Claim}

\theoremstyle{remark}
\newtheorem{remark}{\bf Remark}

\newcommand{\ud}{\mathrm{d}}
\newcommand{\ue}{\mathrm{e}}
\newcommand{\ft}{\mathrm{FT}}
\newcommand{\Gal}{\mathrm{Gal}}

\newcommand{\tr}{\mathrm{tr}}

\newcommand{\rank}{\mathrm{rank}}
\newcommand{\Swan}{\mathrm{Swan}}
\newcommand{\kl}{\mathrm{Kl}}
\newcommand{\Frob}{\mathrm{Frob}}
\newcommand{\Drop}{\mathrm{Drop}}

\DeclareMathOperator{\Mod}{mod}

\renewcommand{\bmod}[1]{\,(\Mod{ #1})}

\newcommand{\ba}{\boldsymbol{a}}

\newcommand{\bh}{\boldsymbol{h}}

\newcommand{\bn}{\boldsymbol{n}}

\newcommand{\bA}{\mathbf{A}}
\newcommand{\bC}{\mathbf{C}}
\newcommand{\bF}{\mathbf{F}}

\newcommand{\bP}{\mathbf{P}}
\newcommand{\bQ}{\mathbf{Q}}
\newcommand{\bR}{\mathbf{R}}
\newcommand{\bZ}{\mathbf{Z}}

\newcommand{\cA}{\mathcal{A}}
\newcommand{\cB}{\mathcal{B}}

\newcommand{\cF}{\mathcal{F}}
\newcommand{\cG}{\mathcal{G}}

\newcommand{\cI}{\mathcal{I}}
\newcommand{\cJ}{\mathcal{J}}
\newcommand{\cK}{\mathcal{K}}
\newcommand{\cL}{\mathcal{L}}

\newcommand{\cP}{\mathcal{P}}

\newcommand{\fc}{\mathfrak{c}}

\newcommand{\fA}{\mathfrak{A}}
\newcommand{\fJ}{\mathfrak{J}}
\newcommand{\fS}{\mathfrak{S}}

\newcommand{\tm}{\tilde{m}}
\newcommand{\tn}{\tilde{n}}

\def\le{\leqslant}
\def\leq{\leqslant}
\def\ge{\geqslant}
\def\geq{\geqslant}

\usepackage{graphicx}
\usepackage{tikz}

\begin{document}

\vglue -2mm

\title[Arithmetic exponent pairs for algebraic trace functions]
{Arithmetic exponent pairs for algebraic trace functions and applications}
\author{Jie Wu}

\address{%
School of Mathematics and Statistics
\\
Qingdao University
\\
308 Ningxia Road
\\
Qingdao
\\
Shandong 266071
\\
China}
\curraddr{%
CNRS, UMR 8050\\
Laboratoire d'Analyse et de Math\'ematiques Appliqu\'ees\\
Universit\'e Paris-Est Cr\'eteil\\
61 Avenue du G\'en\'eral de Gaulle\\
94010 Cr\'eteil cedex\\
France
}
\email{jie.wu@math.cnrs.fr}

\author{Ping Xi \\ (With an appendix by Will Sawin)}

\address{School of Mathematics and Statistics, Xi'an Jiaotong University, Xi'an 710049, China}

\email{ping.xi@xjtu.edu.cn}

\address{ Department of Mathematics, Columbia University, 2990 Broadway, New York, New York 10027, USA}
\email{sawin@math.columbia.edu}

\subjclass[2010]{11T23, 11L05, 11L07, 11N13, 11N36, 11N37, 11M06}

\keywords{$q$-analogue of the van der Corput method, arithmetic exponent pairs, trace functions of $\ell$-adic sheaves, Brun--Titchmarsh theorem, linear sieve}

\begin{abstract} 
We study short sums of algebraic trace functions via the $q$-analogue of the van der Corput method, and develop theory of arithmetic exponent pairs that coincide with the classical case when the moduli have sufficiently good factorizations. As an application, we prove a quadratic analogue of the Brun--Titchmarsh theorem on average, bounding the number of primes $p\leqslant X$ such that
$p^2+1\equiv0\bmod{q}$. The other two applications include a larger level of distribution of divisor functions in arithmetic progressions and 
a sub-Weyl subconvex bound of Dirichlet $L$-functions studied previously by Irving.
\end{abstract}
\vglue -15mm
\maketitle

\setcounter{tocdepth}{1}

\vglue -12mm
\tableofcontents

\vglue -15mm
\section{Introduction}\label{sec:Introduction}

\subsection{Background}
Given a positive integer $q$ and $\Psi:\bZ/q\bZ\rightarrow\bC,$ a non-trivial bound for the average
\begin{align}\label{eq:Psi-average}
S(\Psi; I)= \sum_{n\in I}\Psi(n)
\end{align}
is highly desired in numerous problems in analytic number theory, 
where $I$ is a certain interval. 
The resolution of such a problem usually depends heavily on some tools from Fourier analysis. 
A typical example dates back to the classical estimate for incomplete character sums of P\'olya and Vinogradov, 
who (independently) applied a completing method (or equivalently a certain Fourier expansion) to transform the incomplete sum 
to complete ones and thus obtained non-trivial bounds as long as $|I|\ge q^{1/2+\varepsilon}$. 
An ingenious improvement was later realized by Burgess \cite{Bu62, Bu63}, who was able to work non-trivially 
with shorter sums and in particular, the first subconvexity can be derived for Dirichlet $L$-functions with the conductor aspect.
When $\Psi$ is specialized to some other examples such as additive characters and Kloosterman sums, 
one can also follow the approach of P\'olya and Vinogradov, 
and then succeed roughly in the range $|I|\ge q^{1/2+\varepsilon}$. 

The above $\tfrac{1}{2}$-barrier plays a crucial role in applications, and is highly desirable to be beaten in many instances. As an important example in history, we recall the pioneer work of Hooley \cite{Ho78} on greatest prime factors of cubic polynomials. 
Denote by $P^+(n)$ the greatest prime factor of $n$. To seek a positive constant $\eta$ such that $P^+(n^3+2)>n^{1+\eta}$
infinitely often, Hooley assumed, for some $\gamma>0,$ that
\begin{align}\label{eq:Hooleyconjecture}
\sum_{n\in I,~(n,q)=1} \mathrm{e}\Big(\frac{a\overline{n}}{q}\Big)\ll |I|q^{-\gamma}, \ \ \ (a,q)=1
\end{align}
holds for all intervals $I$ longer than $q^\theta$ for some $\theta<\tfrac{1}{3}$. However, the completing method of P\'olya and Vinogradov barely works for $\theta>\tfrac{1}{2}$. 
The existence of such a positive constant $\eta$ is nowadays known unconditionally due to the efforts of Heath-Brown \cite{HB01}. The approach of Heath-Brown is not devoted to proving a strong estimate such as \eqref{eq:Hooleyconjecture}, 
and instead he modified the Chebyshev--Hooley method so that 
some exponential sums with special features arise. In particular, he was able to allow the modulus $q$ to have suitable factorizations, and an estimate of the shape of \eqref{eq:Hooleyconjecture} can be obtained for such special $q$ by introducing the idea from classical estimates for analytic exponential sums, which is now usually known as {\it $q$-analogue of the van der Corput method}. In what follows we refer this to $q$-vdC for short.

In his breakthrough on bounded gaps between primes, 
Zhang \cite{Zh14} proved a level of distribution of primes in arithmetic progressions that is beyond $\tfrac{1}{2}$. 
A key feather is that he assumes the moduli have only small prime factors and thus allow suitable factorizations. 
He was able to go beyond the P\'olya--Vinogradov barrier in the resultant exponential sums with such special moduli, 
and the underlying idea can also be demonstrated by $q$-vdC.
There are many other examples that benefited a lot from $q$-vdC, and we will try to present a short list in later discussions. 

As in the above instances, one arrives at estimates for certain complete sums over $\bZ/q\bZ$ in the last step, and some tools from algebraic geometry enter the picture to guarantee square-root cancellations. On the other hand,
Fouvry, Kowalski and Michel initiated, from various analytic and geometric points of view, extensive investigations on general trace functions associated to some middle-extension $\ell$-adic sheaves on $\bA_{\bF_p}^1$ (see \cite{FKM14,FKM15,FKM+17} for instance). They are trying to establish a more direct and close relation between analytic number theory and algebraic geometry, where, in most cases, the second one serves as a powerful tool and provides fertile resources for the first, as one can see from the above examples.

\subsection{Plan of this paper}
In this paper, we study $q$-vdC for general trace functions, which are composite in the sense that they are defined by products of trace functions of suitable
$\ell$-adic sheaves on $\bA_{\bF_p}^1$ for a few primes $p$. Roughly speaking, we seek to bound \eqref{eq:Psi-average}
with $\Psi$ specialized to such composite trace functions, containing \eqref{eq:Hooleyconjecture} as a special case.
In fact, this project was initiated by Polymath \cite{Po14} in the improvement to Zhang's work on bounded gaps between primes.
Our observation here allows one to develop a method on arithmetic exponent pairs
analogous to those 
in the classical van der Corput method, from which one can find almost optimal estimates
for such averages as long as the moduli have sufficiently good factorizations.
On the other hand, one can also develop multiple exponent pairs that demonstrate
how the upper bounds depend on each factor of the modulus. We will start from an abstract exponent pair and then produce a series of exponent pairs after applying the $A$- and $B$-processes in $q$-vdC for suitably many times.

Three applications of $q$-vdC are also derived. On one hand, we prove a quadratic analogue of the Brun--Titchmarsh theorem on primes in arithmetic progressions, for which the linear Rosser--Iwaniec sieve plays a fundamental role and thanks to the contributions of Iwaniec \cite{Iw80}, we are able to take full advantage of the well factorizations of remainder terms.
On the other hand, we can, using our arithmetic exponent pairs, recover and improve a large level of Irving on the divisor functions in arithmetic progressions and a sub-Weyl subconvexity for Dirichlet $L$-functions.

The main part of this article and \cite{WX17} were completed in 2016,  and there have been many other related developments since then, including a larger level of ternary divisor function in arithmetic progressions that also requires a bilinear estimate for hyper-Kloosterman sums with smooth moduli (see  \cite{Xi18a}), a quadratic analogue of Titchmarsh divisor problem (see \cite{Xi18b}), a shifted convolution sum for $GL(3)\times GL(2)$ 
by combining arithmetic exponent pairs with Jutila's refinement of the circle method (see \cite{Xi18c}), and an improved lower bound for the number of fundamental solutions to Pell equations with prescribed sizes towards Hooley's conjecture (see  \cite{Xi18d}).
Moreover, some results and ideas are used by Dartyge and Martin \cite{DM19} to study distributions of exponential sums over roots of reducible polynomials.

Due to the special structure of this paper, we cannot state explicitly $q$-vdC and exponent pairs for algebraic trace functions in this section; however, to motivate the readers, we would like to present three applications mentioned as above.

\subsection{Quadratic Brun--Titchmarsh theorem on average} Our first application is devoted to counting primes in arithmetic progressions on average.
Let $q$ be a fixed positive integer and $(a,q)=1,$ we are interested in the counting function
\[\pi(x; q, a)=|\{p\leqslant x:p\equiv a\bmod q\}|.\]
Setting $q\asymp x^\theta$, one may expect,  as $x\rightarrow+\infty,$ that
\begin{align}\label{eq:BT}
\pi(x; q, a)<\{C(\theta)+o(1)\}\frac{1}{\varphi(q)}\frac{x}{\log x}\end{align}
holds for $\theta$ as large as possible with some $C(\theta)>0.$ 
This is called the {\it Brun--Titchmarsh theorem} since Titchmarsh is the first who proved the existence of such $C(\theta)$
 via Brun's sieve. 
By virtue of a careful application of Selberg's sieve, van Lint \& Richert \cite{LR65} showed that
$C(\theta)=2/(1-\theta)$ is admissible for $\theta\in~]0,1[$, uniformly in $(a,q)=1.$
This was later sharpened by Motohashi \cite{Mo74} for $\theta\in~]0, \tfrac{1}{2}]$. 
Iwaniec \cite{Iw82} introduced his description of the bilinear structure of remainder terms in linear sieves \cite{Iw80} 
to this problem, which allowed him to take $8/(6-7\theta)$ for $\theta\in~]\tfrac{2}{5}, \tfrac{2}{3}[$.
The progress becomes slower in this direction and the latest result going beyond $\frac{1}{2}$, to our best knowledge, 
is due to Friedlander \& Iwaniec \cite{FrI97}, 
who may take $C(\theta)=2/(1-\theta)-(1-\theta)^5/2^{12}$ for $\theta\in~]\tfrac{6}{11},1[.$

On the other hand, motivated by the problem on greatest prime factors of shifted primes, Hooley \cite{Ho72, Ho73, Ho75} initiated to bound $\pi(x; q, a)$ from above with an extra average over $q$. The subsequent improvement is due to Iwaniec \cite{Iw82},
who combined Hooley's argument with his bilinear remainder terms in linear sieves.
 It is a common treatment to transform sums over primes to those over integers via sieve methods, and then exponential sums will arise after Poisson summation. One then arrives at Kloosterman sums, so that Weil's bound does this job as argued by Hooley \cite{Ho72} and Iwaniec \cite{Iw82}.
Thanks to the work of Deshouillers \& Iwaniec \cite{DI82a} on the control of sums of Kloosterman sums, one can do much better on the level of linear sieves; see the works by
Deshouillers--Iwaniec \cite{DI84}, Fouvry \cite{Fo84, Fo85a}, Baker--Harman \cite{BH96}, {\it et al}.
However,
due to the use of the ``{\it switching-moduli}" trick, the residue class $a$ is usually assumed to be fixed.

We now extend the classical Brun--Titchmarsh theorems to the quadratic case. 
Let $f\in\bZ[X]$ be a fixed quadratic polynomial, and define
\[\pi_f(x; q) := |\{p\leqslant x: f(p)\equiv0\bmod q\}|.\]
Note that
\begin{align}\label{eq:pi_f}
\pi_f(x; q) = \sum_{\substack{a\bmod{q}\\ f(a)\equiv0\!\bmod{q}}} \pi(x; q, a).
\end{align}
In many situations, the number of solutions to $f(a)\equiv0\bmod{q}$ is usually quite small, say $O(q^{\varepsilon})$, at least when the leading coefficient of $f$ is coprime to $q$, in which case the estimation for $\pi_f(x; q)$ is thus reduced to the classical Brun--Titchmarsh theorem if $q$ is large but fixed. Therefore, our concern is to estimate $\pi_f(x; q)$ with an extra summation over $q$, for which the residue class is no longer fixed when $q$ varies,
and we would encounter quite a different problem from the classical situation.

Let $\ell\geqslant1.$ We restrict ourselves to the special case $f(t)=t^2+1$ and consider
\begin{align}\label{eq:Q_l(X)}
Q_{\ell}(X) := |\{p\leqslant X:p^2+1\equiv0\bmod\ell\}|.
\end{align}
We have the following {\it Quadratic Brun--Titchmarsh Theorem on Average}.
\begin{theorem}\label{thm:Brun-Titchmarsh}
Let $A>0$.
For large $L=X^\theta$ with $\theta\in[\frac{1}{2}, \frac{16}{17}[$, the inequality
\begin{align}\label{Q_l(X)bound}
Q_{\ell}(X)
\leqslant \bigg\{\frac{2}{\gamma(\theta)}+o(1)\bigg\}\frac{\varrho(\ell)}{\varphi(\ell)}\frac{X}{\log X}
\end{align}
holds for all $\ell\in[L,2L]$ with at most $O_A(L(\log L)^{-A})$ exceptions, where
\begin{align}\label{eq:gamma(theta)}
\gamma(\theta) := \begin{cases}
\frac{91-89\theta}{62}  & \text{if $\,\theta\in[\frac{1}{2},\frac{64}{97}[$},
\\\noalign{\vskip 0,5mm}
\frac{86-83\theta}{60}  & \text{if $\,\theta\in[\frac{64}{97},\frac{32}{41}[$},
\\\noalign{\vskip 0,5mm}
\frac{19-18\theta}{14}  & \text{if $\,\theta\in[\frac{32}{41},\frac{16}{17}[$}.
\end{cases}
\end{align}
\end{theorem}


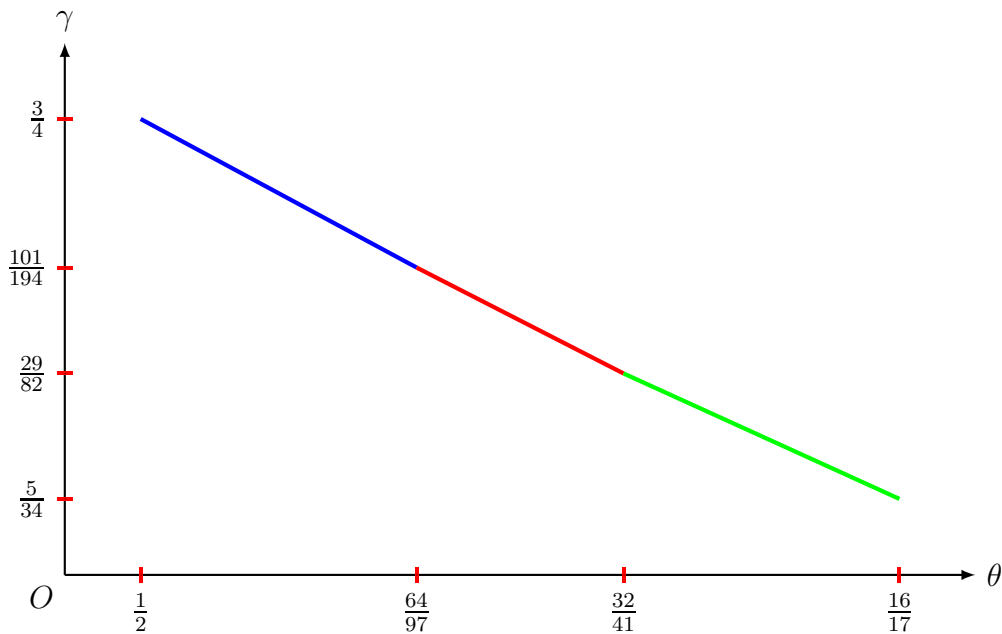
\begin{figure}[h]
  \centering
  \begin{tikzpicture}

    \clip (-1,-1) rectangle (12.5cm, 7.5cm); 

    \coordinate (Origin) at (4,0);
    \coordinate (XAxisMin) at (0,0);
    \coordinate (XAxisMax) at (12,0);
    \coordinate (YAxisMin) at (0,0);
    \coordinate (YAxisMax) at (0,7);

    \coordinate (X1) at (1.00,0);
    \coordinate (X2) at (4.636,0);
    \coordinate (X3) at (7.364,0);
    \coordinate (X4) at (11.00,0);

    \coordinate (Y1) at (0, 6);
    \coordinate (Y2) at (0, 4.045);
    \coordinate (Y3) at (0, 2.651);
    \coordinate (Y4) at (0, 1.00);

    \coordinate (P1) at (1, 6);
    \coordinate (P2) at (4.636, 4.045);
    \coordinate (P3) at (7.364, 2.651);
    \coordinate (P4) at (11.00, 1.00);

    \draw [thick,-latex] (XAxisMin) -- (XAxisMax) node [right]{\(\theta\)};
    \draw [thick,-latex] (YAxisMin) -- (YAxisMax) node [above]{\(\gamma\)};

  \node[inner sep=0mm, minimum height=0.4mm, minimum width=2mm, red,fill] at (Y1) {};
  \node[inner sep=0mm, minimum height=0.4mm, minimum width=2mm, red,fill] at (Y2) {};
  \node[inner sep=0mm, minimum height=0.4mm, minimum width=2mm, red,fill] at (Y3) {};
  \node[inner sep=0mm, minimum height=0.4mm, minimum width=2mm, red,fill] at (Y4) {};
 
  \node[inner sep=0mm, minimum height=2mm, minimum width=0.4mm, red,fill] at (X1) {};
  \node[inner sep=0mm, minimum height=2mm, minimum width=0.4mm, red,fill] at (X2) {};
  \node[inner sep=0mm, minimum height=2mm, minimum width=0.4mm, red,fill] at (X3) {};
  \node[inner sep=0mm, minimum height=2mm, minimum width=0.4mm, red,fill] at (X4) {};

  \draw (0,0) -- (0,0) node [below left]{\(O\)};
  \draw (X1) -- (X1) node [below=1mm]{\(\frac{1}{2}\)};
  \draw (X2) -- (X2) node [below=1mm]{\(\frac{64}{97}\)};
  \draw (X3) -- (X3) node [below=1mm]{\(\frac{32}{41}\)};
  \draw (X4) -- (X4) node [below=1mm]{\(\frac{16}{17}\)};

  \draw (Y1) -- (Y1) node [left=1mm]{\small \(\tfrac{3}{4}\)};
  \draw (Y2) -- (Y2) node [left=1mm]{\small \(\tfrac{101}{194}\)};
  \draw (Y3) -- (Y3) node [left=1mm]{\small \(\tfrac{29}{82}\)};
  \draw (Y4) -- (Y4) node [left=1mm]{\small \(\tfrac{5}{34}\)};
  
   \draw [ultra thick, blue] (P1) -- (P2) ;
   \draw [ultra thick, red] (P2) -- (P3) ;
   \draw [ultra thick, green] (P3) -- (P4) ;

  
\end{tikzpicture}
\caption{Graph of \(\gamma(\theta)\) as a function of \(\theta\).}
\label{pic_unmoinseta}
\end{figure}



The proof of Theorem \ref{thm:Brun-Titchmarsh} can be summarized as follows. 
The linear sieve of Rosser--Iwaniec applies to the prime variable in $Q_{\ell}(X),$
and a routine application of Fourier analysis will lead us to the Weyl sum
\begin{align}\label{eq:rho(n,l)}
\varrho_n(\ell) 
:= \sum_{\substack{a\bmod{\ell}\\a^2+1\equiv0\!\bmod{\ell}}} \mathrm{e}\Big(\frac{an}{\ell}\Big).
\end{align}
A trivial bound reads $|\varrho_n(\ell)|\leqslant \varrho_0(\ell) = \varrho(\ell)\ll \ell^\varepsilon$ for any $\varepsilon>0,$ 
which means one cannot expect any power-savings if each $\varrho_n(\ell)$ is taken into account individually. Fortunately, we may follow the approaches of Hooley \cite{Ho67} and Deshouillers--Iwaniec \cite{DI82b}, transforming $\varrho_n(\ell)$ to another exponential sums by appealing to Gau{\ss}'s theory of representation of numbers by binary quadratic forms (see Lemma \ref{lm:Gausslemma} below), and a considerable cancellation is possible when summing over $\ell$.

The above-mentioned exponential sums in \cite{Ho67,DI82b} are both Kloosterman sums, and cancellations among such sums can be controlled by virtue of the spectral theory of automorphic forms (see \cite{DI82a}). In our current situation, due to the application of linear sieves before Fourier analysis, we will be led to some algebraic exponential sums that are not perfectly Kloosterman sums, so that we have to go back to the original approach of Hooley \cite{Ho72}. However, we may invoke the work of Iwaniec \cite{Iw80} on the well-factorable remainder terms in linear sieves. In such way, the moduli of the resultant exponential sums allow \textit{suitable factorizations}. We can thus employ $q$-vdC to capture cancellations although the sums are quite short, the underlying ideas of which are the key observations in our arguments. 

Theorem \ref{thm:Brun-Titchmarsh} is in fact motivated by some arithmetic problems 
concerning quadratic polynomials at prime arguments. 
In another joint work \cite{WX17}, we consider the greatest prime factors and almost prime values of $p^2+1$, 
as approximations to the conjecture that any given quadratic irreducible polynomial can capture infinitely many prime values at prime arguments, provided that there are no obvious obstructions.
One will see that our methods allow us to improve significantly corresponding results in literature.

Before closing this subsection, we would like to mention that Theorem \ref{thm:Brun-Titchmarsh} can be extended to general quadratic irreducible polynomials of fixed discriminants.

\subsection{Divisor functions in arithmetic progressions}
For $(a,q)=1$, define
\begin{align*}
D(X; q, a) := \sum_{\substack{n\leqslant X\\n\equiv a\!\bmod{q}}} \tau(n).
\end{align*}
Put $q\asymp X^\theta.$ It is important to evaluate $D(X; q, a)$ asymptotically
with $\theta$ as large as possible. As a direct application of Weil's bound for Kloosterman sums (together with Fourier analysis), one has, as $X\rightarrow+\infty$,
\begin{align}\label{eq:tauAPs-asymptotic}
D(X; q, a)
= \{1+o(1)\}\frac{1}{\varphi(q)}\sum_{\substack{n\leqslant X\\(n,q)=1}}\tau(n)\end{align}
uniformly in $a$ for any $\theta<\tfrac{2}{3}$. 
This was independently obtained by Selberg and Hooley, and is still the best known record for an arbitrary modulus $q$. 
It is reasonable to expect that \eqref{eq:tauAPs-asymptotic} should hold for all $\theta<1.$

For almost all $q\asymp X^\theta,$ Fouvry \cite{Fo85b} succeeded in the case $\tfrac{2}{3}<\theta<1$, but he has to fix the residue class $a$. 
The gap around $\theta\approx \tfrac{2}{3}$ was covered by Fouvry and Iwaniec \cite{FoI92} 
for almost all $q$ satisfying certain factorizations.

Quite recently, Irving \cite{Ir15} picked up this problem when $q$ has sufficiently good factorizations, 
so that $q$-vdC applies. His main theorem could be formulated as follows. 
If $q\asymp X^\theta$ is squarefree and has only prime factors not exceeding $X^\eta,$
then \eqref{eq:tauAPs-asymptotic} holds for all $\theta<\tfrac{2}{3}+\varpi$ with $246\varpi+18\eta<1.$ 
In particular, if $\eta$ is sufficiently small, i.e., $q$ is smooth enough,
he can take \[\theta=\frac{2}{3}+\frac{1}{246}+O(\eta).\]

The arithmetic exponent pairs in Section \ref{sec:q-vdC} 
allow us to obtain a slightly larger admissible value of $\theta.$ Note that Irving used the exponent pair $BA^3(\frac{1}{2},\frac{1}{2})=(\frac{11}{30},\frac{16}{30})$.

\begin{theorem}\label{thm:tauAPs}
Suppose that $q\asymp X^\theta$ is squarefree and has only prime factors not exceeding $q^\eta$ 
with $\eta>0$ sufficiently small. 
Then \eqref{eq:tauAPs-asymptotic} holds for all $\theta$ with
\[\theta\le \frac{2}{3}+\frac{1}{232}\]
uniformly in $(a,q)=1$.
\end{theorem}

\subsection{Subconvexity of Dirichlet $L$-functions to smooth moduli}
Given a positive integer $q$ and a primitive character $\chi\bmod{q},$ we are interested in obtaining subconvex bounds for Dirichlet $L$-functions, i.e., we expect, for some $\delta>0,$ that
\begin{align}\label{eq:L-subconvexity}
L(\tfrac{1}{2}, \chi)\ll q^{\frac{1}{4}-\delta}.\end{align}

Burgess \cite{Bu63} proved that any $\delta<\frac{1}{16}$ is admissible in \eqref{eq:L-subconvexity}, 
and a Weyl bound asserts that any $\delta<\frac{1}{12}$ should be accessible, 
as a $q$-analogue of subconvexity of the Riemann zeta function. 
The Weyl bound was already achieved in a few cases: Heath-Brown \cite{HB78} succeeded under the assumption that $q$ admits suitable factorizations; 
Conrey and Iwaniec \cite{CI00} solved the case that $q$ is a prime and $\chi$ is quadratic by quite a different method, which was recently refined and generalized by Young \cite{Yo17} and Petrow and Young
\cite{PY20}.

Heath-Brown's argument relies on the ideas of $q$-vdC as we have mentioned above, and this was developed further by Irving \cite{Ir16}, going beyond Weyl's barrier if $q$ has sufficiently good factorizations. In particular, if $q$ is squarefree and has only prime factors not exceeding $q^\eta$ for some small $\eta>0$, then one can take
\begin{align*}
\delta=\frac{7}{82}+O(\eta)
\end{align*}
in \eqref{eq:L-subconvexity}. 
As $\eta$ becomes sufficiently small, he goes beyond the Weyl bound since $\tfrac{1}{12}=\frac{7}{84}<\tfrac{7}{82}$. 
This coincides with the bound for $\zeta(\frac{1}{2}+\mathrm{i}t)$ in the $t$-aspect, 
derived from the classical exponent pair $(\frac{11}{82},\frac{57}{88})$.

The above assumption of Irving on $q$ just falls into the application of arithmetic exponent pairs developed in this paper. More precisely, we obtain the following improvement.

\begin{theorem}\label{thm:subconvexity}
Suppose that $q$ is squarefree and has only prime factors not exceeding $q^\eta$ with $\eta>0$ sufficiently small. 
Then, for any
primitive Dirichlet character $\chi\bmod{q}$, we have
\begin{align*}
L(\tfrac{1}{2}, \chi)\ll  q^{\frac{1}{4}-0.085489}.\end{align*}
\end{theorem}

Note that $\tfrac{7}{82}\approx 0.085365.$ The improvement is rather slight, however the proof is clear as an immediate consequence of arithmetic exponent pairs.

This paper is organized as follows: the terminology of trace functions will be introduced in Section \ref{sec:tracefunction}, and in Section \ref{sec:q-vdC}
we develop the method of arithmetic exponent pairs. Theorems \ref{thm:Brun-Titchmarsh}, \ref{thm:tauAPs} and \ref{thm:subconvexity} will be proved in Sections \ref{sec:quadraticBT} and \ref{sec:tauAPs-subconvexity}. 
The Mathematica codes can be found at \url{http://gr.xjtu.edu.cn/web/ping.xi/miscellanea} or requested from the authors.

\subsection*{Notation and convention} 
As usual, $\tau$, $\varphi$ and $\Lambda$ denote the divisor, Euler and von Mangoldt functions, respectively. 
The variable $p$ is reserved for prime numbers. For a real number $x,$ denote by $[x]$ its integral part. Denote by $\varrho(\ell)$ the number of solutions to the congruence equation $n^2+1\equiv0 \bmod{\ell}.$  For a given positive integer $n$, we adopt the convention $n^\infty$ to mean a sufficiently large power of $n$, so that 
\[(m,n^\infty)=\prod_{p^\nu\| m,~p\mid n}p^\nu.\]
Given a positive integer $n,$ we introduce the new arithmetic function
\begin{align}\label{eq:Xi}
\Xi(n) := \prod_{p^\nu\|n, \, \nu\geqslant3} p^\nu.
\end{align}

For a function $f$ defined over $\bZ/q\bZ,$ the Fourier transform is defined by
\begin{align*}
\widehat{f}(y) := \frac{1}{\sqrt{q}}\sum_{a\in\bZ/q\bZ}f(a)\mathrm{e}\Big(\frac{-ya}{q}\Big)
\end{align*}
where $\mathrm{e}(t):=\mathrm{e}^{2\pi\mathrm{i}t}$.
For each $h\in\bZ$ and all $x\in\bZ/q\bZ$, define the difference
\begin{equation}\label{def:Deltahf}
\Delta_h(f)(x) := f(x)\overline{f(x+h)}.
\end{equation}
For a function $g\in L^1(\bR)$, its Fourier transform is defined by
\begin{align*}
\widehat{g}(y) := \int_\bR g(x) \mathrm{e}(-yx)\ud x.
\end{align*}

The multiplicative inverse $\overline{x}$ of $x$ should be defined with respect to some specialized modulus $c$; i.e., $\overline{x}x\equiv1\bmod{c}.$ Moreover, when $\overline{x}$
 appears in fractions, the modulus will be referred implicitly to the denominator, which is assumed to be coprime to $x$ as can be checked on each occasion.

We use $\varepsilon$ to denote a small positive number, which might be different at each occurrence; we also write $X^\varepsilon \log X\ll X^\varepsilon.$ The notation $n\sim N$ means $N<n\leqslant2N.$  

\subsection*{Acknowledgements} 
The authors are grateful to \'Etienne Fouvry and Philippe Michel for their kind suggestions and 
to C\'ecile Dartyge for pointing out an error in an earlier version of Lemma \ref{lm:Gausslemma}. Sincere thanks are also due to Will Sawin for his appendix which makes the method developed in this paper more applicable in analytic number theory.
Both of Jie Wu and Ping Xi are supported in part by NSFC (No. 11971370) and PRC NSFC-CNRS (No. 11911530227/202175).
Jie Wu is supported in part by NSFC (No. 12071375) and Ping Xi is supported in part by NSFC (No. 11601413) and The Young Talent Support Plan in Xi'an Jiaotong University. 

\smallskip

\section{Basics on algebraic trace functions}\label{sec:tracefunction}

This section is devoted to the terminology on trace functions of $\ell$-adic sheaves on $\bA_{\bF_p}^1$ 
following the manner of Fouvry, Kowalski and Michel \cite{FKM14, FKM15}, 
and $\ell$-adic Fourier transforms will also be discussed after Laumon \cite{La87} and Katz \cite{Ka90}.

\subsection{Trace functions}
Let $p$ be a prime and $\ell\neq  p$ an auxiliary prime, and fix an isomorphism $\iota : \overline{\bQ}_\ell\rightarrow\bC$. The functions $K(x)$ modulo $p$ that we
consider are the trace functions of suitable constructible sheaves on $\bA^1_{\bF_p}$
evaluated at $x\in\bF_p$. To
be precise, we will consider middle-extension sheaves on $\bP^1_{\bF_p}$ 
and we refer to the following definition after Katz \cite[Section 7.3.7]{Ka88}.

\begin{definition}[Trace functions]\label{def:tracefunction}
Let $\cF$ be an $\ell$-adic middle-extension sheaf pure of weight zero, 
which is lisse on an open set $U$. The trace function associated to $\cF$ is defined by
\begin{align*}
K:x\in\bF_p\mapsto\iota(\tr(\Frob_x\mid V_\cF)),
\end{align*}
where $\Frob_x$ denotes the geometric Frobenius at $x\in\bF_p,$ and $V_\cF$ is a finite dimensional $\overline{\bQ}_\ell$-vector space, which is corresponding to a continuous finite-dimensional Galois representation and unramified at every closed point $x$ of $U.$
\end{definition}

We need an invariant to measure the geometric complexity of a trace function, which can be given by some numerical invariants of the underlying sheaf.
\begin{definition}[Conductor] \label{def:conductor} 
For an $\ell$-adic middle-extension sheaf $\cF$ on $\bP^1_{\bF_p}$ of rank $\rank(\cF)$,
we define the $($analytic$)$ conductor of $\cF$ to be
\begin{align*}  
\fc(\cF) := \rank(\cF) + \sum_{x\in S(\cF)} (1+\Swan_x(\cF)),
\end{align*}
where $S(\cF)\subset\bP^1(\overline{\bF}_p)$ denotes the $($finite$)$ set of singularities of $\cF$, 
and $\Swan_x(\cF)$ $(\geqslant 0)$ denotes the Swan conductor of $\cF$ at $x$ $($see {\rm \cite{Ka80}}$).$
\end{definition}

We never lack for practical examples of trace functions in modern analytic number theory. For instance,
\begin{itemize}
\item Let $f\in\bF_p(X)$ be a rational function, and $\psi$ a primitive additive character on $\bF_p$, then $\psi(f(x))$ is a trace function of an $\ell$-adic middle-extension sheaf, which is taken to be zero when meeting a pole of $f$ at $x$. More precisely, one can show that there exists an $\ell$-adic middle-extension sheaf modulo $p$, denoted by
$\cL_{\psi(f)},$ such that $x\mapsto\psi(f(x))$ is the trace function of $\cL_{\psi(f)}.$ The conductor can be bounded in terms of the degree of $f$, independent of $p$.

\item  
Let $f\in\bF_p(X)$ be a rational function, and $\chi$ a multiplicative character of order $d >1$. 
If $f$ has no pole or zero of order divisible by $d$, 
then one can show that there exists an $\ell$-adic middle-extension
sheaf, denoted by $\cL_{\chi(f)}$, such that $x\mapsto\chi(f(x))$ is the trace function of $\cL_{\chi(f)}$. 
The conductor can also be bounded in terms of the degree of $f$, independent of $p$.

\item 
Another example is the following {\it normalized} hyper-Kloosterman sum 
defined, for any fixed positive $k,$ by
\begin{align*}
\kl_k(\cdot,p) : x\mapsto p^{-(k-1)/2}
\mathop{\sum\cdots\sum}_{\substack{x_1, \, \dots, \, x_k\in\bF_p\\ x_1\cdots x_k=x}}
\mathrm{e}\Big(\frac{x_1+\cdots+x_k}{p}\Big).
\end{align*}
Note that $\kl_k(0,p)=(-1)^{k-1}p^{-(k-1)/2}.$
In particular, we have $\kl_1(x,p)=\mathrm{e}(x/p)$, and
$\kl_2(x,p)$ normalizes the classical Kloosterman sum at the invertible point $x\in\bF_p^\times$. According to Deligne, there exists an $\ell$-adic middle-extension sheaf $\cK\ell_k$, called a Kloosterman sheaf, such that
\begin{align*}K_{\cK\ell_k}(x)=\kl_k(x,p)
\quad
\text{for all}\ \ x\in\bF_p^\times.
\end{align*}
Such a sheaf was constructed by Deligne \cite{De80}, and extensively studied by Katz \cite{Ka80, Ka88}. 
Again according to Deligne, $\cK\ell_k$ is geometrically irreducible and is of rank $k$,  the conductor of which is bounded by $k+3$.
\end{itemize}

Let $q$ be a squarefree number. 
What will be concerned with is a {\it composite} trace function $K$ modulo $q$, given by the product 
\[K(n)=\prod_{p|q}K_p(n),\]
where $K_p$ is a trace function associated to some $\ell$-adic middle-extension sheaf
on $\bP_{\bF_p}^1$.
We adopt the convention that $K(n)=1$ for all $n$ if $q=1.$
In practice, the value of $K_p(n)$ may depend on the complementary divisor $q/p.$ 
Many definitions involving $K_p$ and $\cF_p$ can be formally moved to $K$ with $q$ composite; see Definitions \ref{def:admissiblesheaf} and \ref{def:amiablesheaf} for instance.

In the study of trace functions, especially on their analytic properties, one usually needs to control the conductors independently of $p$, as in the above examples. On the other hand, the following Riemann Hypothesis, proved by Deligne \cite{De80},  plays an essential role in the practical device, demonstrating the quasi-orthogonality of trace functions of $\ell$-adic sheaves (see \cite[Corollary 6.6]{Po14} for instance).

\begin{proposition}[Quasi-Orthogonality]\label{prop:RH}
Suppose $\cF_1,\cF_2$ are two admissible $($in the sense of Definition $\ref{def:admissiblesheaf}$ below$)$ sheaves on $\bP_{\bF_p}^1$, and $K_1,K_2$ are the associated trace functions, respectively. If $\cF_1,\cF_2$ have no common geometrically irreducible components, then there exists an absolute constant $C>0$ such that
\begin{align*}
\Bigg|\sum_{x\in\bF_p}K_1(x)\overline{K_2(x)}\Bigg|
\leqslant C\cdot\fc(\cF_1)^4\fc(\cF_2)^4\sqrt{p}.
\end{align*}
\end{proposition}

\subsection{Amiable trace functions for $q$-vdC}
Given an average $S(\Psi; I)$ in (\ref{eq:Psi-average}) with $\Psi$ being specialized to some trace function $K\bmod q$, the $A$-process in $q$-vdC (see Section \ref{sec:q-vdC} below) usually produces certain sums involving the difference $\Delta_h(K_p)$ for $p\mid q$ and some $h\in\bZ$. 
Observe that if
\[K_p(x)=\psi(ax^2+bx)
\quad\text{with}\quad
a\in\bF_p^\times,\]
one has $\Delta_h(K_p)(x)=\psi(-2ahx-bh-ah^2)$, and the resultant sum reveals no cancellation after one more $A$-process since the summand becomes a constant function.
As we will see, this phenomenon is essentially the only obstruction to square-root
cancellations. We thus need to determine when a trace function is suitable for the purpose of our analysis, to which we mean {\it amiable}. The similar arguments first appeared in \cite[Section 6]{Po14}, but a different convention was used therein.

We first formulate the admissibility.
\begin{definition}[Admissible sheaf]\label{def:admissiblesheaf}
An admissible
sheaf $\cF_p$ over $\bF_p$ is a middle-extension sheaf on $\bP_{\bF_p}^1$
which is pointwise pure of weight $0$ $($in the sense of Deligne \cite{De80}$).$ A trace function $K_p\bmod p$ said to be admissible if the corresponding sheaf $\cF_p$ is admissible.
A composite trace function $K\bmod{q}$ is said to be admissible, if the reduction $K_p$ is admissible for each $p\mid q.$
\end{definition}

Given an admissible sheaf, we say it is geometrically irreducible if the corresponding sheaf is geometrically irreducible.

\begin{definition}[Amiable sheaves and trace functions]\label{def:amiablesheaf}
Let $d$ be a non-negative integer, $p\geqslant2$ a prime and $q\geqslant1$ a squarefree number.  An admissible sheaf $\cF_p$ over $\bF_p$ is said to be $d$-amiable if it is geometrically isotypic and no geometrically irreducible component is geometrically isomorphic to an Artin--Schreier sheaf of the form $\cL_{\psi(P)},$ where $P\in\bF_p[X]$ is a polynomial of degree $\leqslant d.$ In such case, we also say the associated trace function $K_p$ is $d$-amiable. 

A composite trace function $K\bmod{q}$ is said to be compositely $d$-amiable if for each $p\mid q,$ $K_p$ can be decomposed into a sum of $d$-amiable trace functions, in which case we also say the corresponding sheaf $\cF:=(\cF_p)_{p\mid q}$ is compositely $d$-amiable.

In addition, a sheaf $($or its associated trace function$)$ is said to be $($compositely$)$ $\infty$-amiable if it is $($compositely$)$ amiable for any fixed $d\geqslant1.$

\end{definition}

\begin{remark}
For an admissible sheaf $\cF$ over $\bF_p,$ it is 1-amiable if and only if it is a geometrically isotypic Fourier sheaf. By Lemma \ref{lm:fouriertransform}, the Fourier transform $\ft_\psi(\cF)$ of an isotypic Fourier sheaf $\cF$ is also an isotypic Fourier sheaf, and is thus definitely 1-amiable.
\end{remark}

\begin{remark}
Given an admissible sheaf $\cF$, one has to determine if it is $\infty$-amiable when applying $q$-vdC along with quite a few iterations. A sufficient condition is that each geometrically irreducible component of $\cF$ is irreducible
of rank $\geqslant2$, or particularly if $\cF$ itself is geometrically irreducible of rank $\geqslant2$.
\end{remark}

\begin{remark}
According to Deligne, an admissible sheaf $\cF$ over $\bF_p$ can be decomposed into a direct sum of arithmetically irreducible components.
Regarding the geometric reducibility, each component is either geometrically isotypic or induced from a representation of $\Gal(K^{\mathrm{sep}}/k.K)$ for $k$ some finite extension of $\bF_p$ with $K=\bF_p(X)$. For the latter case, the associated trace function to such component vanishes identically. In the above sense, it suffices to consider geometrically isotypic sheaves, or more basically, geometrically irreducible sheaves.

Therefore, a compositely $d$-amiable $K\bmod q$ is a sum of at most $O(\fc^{\omega(q)})$ trace functions $K_i\bmod q$, and for each $p\mid q,$ the reduction $K_{i,p}$ of $K_i$ modulo $p$ is $d$-amiable. Here $\fc$ is a constant that controls the conductor of each sheaf associated to $K_{i,p}$. Note that $\fc^{\omega(q)}\ll q^\varepsilon$ for any $\varepsilon>0,$ and this factor is allowed in our subsequent estimates.
\end{remark}

To apply the $A$- and $B$-processes inductively, one needs to check that the amiability can be stable in a certain sense.
For an $\ell$-adic middle-extension sheaf $\cF$ on $\bP_{\bF_p}^1$, denote by $[+a]^*\cF$ the 
pullback of $\cF$ under the additive shift $n\mapsto n+a$, where $a\in\bF_p.$ We also write
$\widecheck{\cF}$ for the the middle-extension dual of $\cF,$ i.e., given a dense open set $j:U\hookrightarrow \bP^1$ where $\cF$ is lisse, we have
\[\widecheck{\cF}=j_*(j^*(\cF)'),\]
where $'$ denotes the lisse sheaf of $U$ associated to the contragredient of the representation of the fundamental group of $U$ corresponding to $j^*\cF.$ If $\cF$ is
an admissible sheaf, the trace function of $\widecheck{\cF}$ appears as the complex conjugate of $K$, a trace function of $\cF.$ In $A$-process, we expect (at least) that $n\mapsto \Delta_a(K)(n)=K(n)\overline{K(n+a)}$ is also an amiable trace function if $a\in\bF_p^\times.$ This of course requires the independence between the sheaves $\cF$ and $[+a]^*\cF$.

\begin{lemma}\label{lm:A-amiability}
Let $d$ be a positive integer and $p>d$. Suppose $\cF$ is a $d$-amiable admissible sheaf over $\bF_p$ with $\fc(\cF)\leqslant p.$ Then, for each $a\in\bF_p^\times,$ the sheaf $[+a]^*\cF\otimes\widecheck{\cF}$ is compositely $(d-1)$-amiable with
\begin{align}\label{eq:sheaftensor-conductor}
\fc([+a]^*\cF\otimes\widecheck{\cF})\leqslant 5\fc(\cF)^4.
\end{align}
More precisely, the trace function of $[+a]^*\cF\otimes\widecheck{\cF}$ can be decomposed into the sum of $\leqslant 5\fc(\cF)^4$ of trace functions, each of which is $(d-1)$-amiable and
has a conductor at most $5\fc(\cF)^4.$
\end{lemma}
\proof 
The amiability of $[+a]^*\cF\otimes\widecheck{\cF}$ follows immediately from \cite[Theorem 6.15]{Po14} and \cite[Proposition 8.3]{FKM15}. The inequality in \eqref{eq:sheaftensor-conductor} is a special case of \cite[Proposition 8.2]{FKM15}.
\endproof

\begin{remark}
 For any $\infty$-amiable admissible sheaf and sufficiently large prime $p$, we conclude from Lemma \ref{lm:A-amiability} that the geometrically isotypic components of $[+a]^*\cF\otimes\widecheck{\cF}$, for each $a\in\bF_p^\times$, are also $\infty$-amiable.
\end{remark}

\subsection{$\ell$-adic Fourier transforms}
As we will see, the $B$-process is an application of Poisson summation essentially, and the Fourier transform of the original trace function will appear. For the purpose of a reasonable interpretation, we introduce the $\ell$-adic Fourier transform over $\bF_p$ starting from a reformulation on the usual Fourier transform over $\bZ/p\bZ$ (up to an opposite sign). Such transforms are well-defined for Fourier sheaves (see also Katz \cite[Definition 8.2.2]{Ka88}).

\begin{definition}[Fourier sheaf]\label{def:fouriersheaf}
An admissible sheaf $\cF$ on $\bP_{\bF_p}^1$ is called a Fourier sheaf if no
geometrically irreducible component is geometrically isomorphic to an Artin--Schreier
sheaf  $\cL_\psi$  attached to some additive character $\psi$ of $\bF_p.$
\end{definition}

We collect the properties of Fourier transforms of Fourier sheaves due to Deligne \cite{De80}, Laumon \cite{La87}, 
Brylinski \cite{Br86}, Katz \cite{Ka88, Ka90} and Fouvry--Kowalski--Michel \cite{FKM15}.

\begin{lemma}\label{lm:fouriertransform}
Let $\psi$ be a non-trivial additive character of $\bF_p$ and $\cF$ a Fourier sheaf on $\bP_{\bF_p}^1.$  Then there exists an $\ell$-adic sheaf $\cG=\ft_\psi(\cF)$
called the Fourier transform of $\cF$, which is also an $\ell$-adic Fourier sheaf, with
the property that 
\begin{align}\label{eq:ell-adicFourier}
K_{\ft_\psi(\cF)}(y)
= \ft_\psi(K_\cF)(y)
:= -\frac{1}{\sqrt{p}}\sum_{x\in\bF_p}K_\cF(x)\psi(yx).
\end{align}
Furthermore, we have
\begin{itemize}

\item 
The sheaf $\cG$ is geometrically irreducible, or geometrically isotypic, if and only if $\cF$ is;

\item 
The Fourier transform is involutive, in the sense that we have a canonical arithmetic isomorphism
$\ft_\psi(\cG)\simeq[\times(-1)]^*\cF,$
where $[\times(-1)]^*$ denotes the pull-back by the map $x\mapsto-x;$

\item 
We have
\begin{align}\label{eq:cond-fouriertranform}
\fc(\ft_\psi(\cF))\leqslant10\fc(\cF)^2.\end{align}
\end{itemize}

\end{lemma}

\proof The last claim was proved by Fouvry, Kowalski and Michel \cite{FKM15} using the theory of local Fourier transforms developed by Laumon \cite{La87}, and the others can be found for instance in \cite[Theorem 8.4.1]{Ka90}.
\endproof

The inequality \eqref{eq:cond-fouriertranform} is essential in analytic applications, since it implies that if $p$ varies but
$\cF$ has a bounded conductor, so does the Fourier transform of $\cF$.

In the subsequent applications of $B$-process, one has to determine if the Fourier transform of a given Fourier sheaf is amiable, and sometimes the rank does this job. 
 According to Katz \cite[Lemma 7.3.9]{Ka90}, one has
\begin{align}\label{eq:FT-rank}
\rank(\ft_\psi(\cF))=\sum_{\lambda}\max(0,\lambda-1)+\sum_{x}(\Swan_x(\cF)+\Drop_x(\cF)),\end{align}
where $\lambda$ runs over the breaks of $\cF(\infty)$ and $x$ over the singularities of $\cF$ in $\overline{\bF}_p$. Here $\Drop_x(\cF)=\rank(\cF)-\dim(\cF_x)$, which is at least 1 at each singularity. 

For $f_1,f_2\in\bF_p[X]$ with $\deg(f_1)<\deg(f_2)<p$, the Artin--Schreier sheaf $\cF:=\cL_{\psi(f_1/f_2)}$ is of rank 1 and $\infty$-amiable for any primitive additive character $\psi$ of $\bF_p$. Since $\cF$ has at least one singularity in $\bA_{\bF_p}^1$, at which the Swan conductor is at least one, it then follows that the rank of the Fourier transform of $\cF$ is at least two, i.e.,
\[\rank(\ft_\psi(\cF))\geqslant2,\]
so that $\ft_\psi(\cF)$ is also $\infty$-amiable.

Following the above arguments, we may find several examples of trace functions that are $\infty$-amiable in the sense of Definition \ref{def:amiablesheaf}.
\begin{itemize}
\item $K_p(n)=\psi(f_1(n)\overline{f_2(n)})$, where $\psi$ is a primitive additive character, $f_1,f_2\in\bF_p[X]$, $\deg(f_1)<\deg(f_2)<p$;

\item $K_p(n)=\chi(f(n))\psi(g(n))$, where $\chi$ is a primitive multiplicative character mod $p$, $\psi$ is not necessarily primitive, $f,g$ are rational functions
and $f$ is not a $d$-th power of another rational function with $d$ being the order of $\chi$; we remark here that the linear case $n\mapsto\chi(n)\psi(n)$ is $\infty$-amiable because $\cL_\chi\otimes\cL_\psi$ is geometrically irreducible and has one singularity at $0.$

\item $K_p(n)=\kl_k(n,p)$ as a normalized hyper-Kloosterman sum of rank $k\geqslant2$;

\item The Fourier transforms of the above examples.
\end{itemize}

In view of the following lemma, we always arrive at
sheaves which are suitably amiable after one $B$-process, which allows us to apply the $A$-process further.

\begin{lemma}\label{lm:B-amiability}
Suppose $r\geqslant1, d\geqslant2$ and $a\in\bF_p^\times$. If $\cF$ is a compositely $d$-amiable sheaf on $\bP_{\bF_p}^1$ of rank $r$ with $\fc(\cF)\leqslant  p.$ Denote by $\cG$ the Fourier transform of $[+a]^*\cF\otimes\widecheck{\cF}$.
\begin{itemize}
\item If $r=1,$ then $\cG$ is compositely $1$-amiable when $d=2,$ and is compositely $\infty$-amiable when $d\geqslant3.$

\item If $r\geqslant2$ and $\cF$ is geometrically isotypic, then $\cG$ is compositely $2$-amiable. Moreover, for a given $a\in\bF_p^\times,$ if $[+a]^*\cF\otimes\widecheck{\cF}$ is geometrically irreducible, then $\cG$ is 
geometrically irreducible and $r^2$-amiable.
\end{itemize}
\end{lemma}

\proof 
Suppose $\cG$ contains $\cL_{\psi(g)}$ for some $g\in\bF_p[X]$ of degree $k\geqslant2$. 
Then $\ft_\psi(\cL_{\psi(g)})$ should be contained in $[+a]^*\cF\otimes\widecheck{\cF}$ up to a geometric isomorphism.
Note that $\ft_\psi(\cL_{\psi(g)})$ is lisse on $\bA_{\bF_p}^1$ and of rank $k-1$ by
\cite[Section 7.12]{Ka90}. Moreover, Lemma \ref{lm:A-amiability} implies that
$[+a]^*\cF\otimes\widecheck{\cF}$ is compositely $(d-1)$-amiable.

We first consider the case $r=1,$ which implies that $[+a]^*\cF\otimes\widecheck{\cF}$ is geometrically irreducible of rank 1, and $\ft_\psi(\cL_{\psi(g)})\simeq[+a]^*\cF\otimes\widecheck{\cF}.$ If $d=2$, then it is possible that $g$ is of degree $k=2$, in which case one finds
$\cG$ should be 1-amiable.
If $d\geqslant3$, then we must have $k\geqslant3,$ and $\ft_\psi(\cL_{\psi(g)})$ is of rank at least 2, which implies that the above $k$ of finite values cannot exist, i.e., 
$\cG$ must be $\infty$-amiable.

We now come to the case $r\geqslant2$, for which we find $[+a]^*\cF\otimes\widecheck{\cF}$ is compositely $\infty$-amiable by Lemma \ref{lm:A-amiability}. There is no reason that $[+a]^*\cF\otimes\widecheck{\cF}$ must be geometrically isotypic. However, it should be semisimple, and can be decomposed into isotypic components. Denote by $\cF_1$ the component that is geometrically isomorphic to $\ft_\psi(\cL_{\psi(g)})$. If $\cF_1$ is of rank $r_1\geqslant2,$ we are done since $\ft_\psi(\cL_{\psi(g)})$ is of rank $k-1\geqslant2.$
If $\cF_1$ is of rank $1$, then so is $\ft_\psi(\cL_{\psi(g)})$, which implies that $g$ is a quadratic polynomial in $\bF_p[X].$ This, however, contradicts the amiability of $[+a]^*\cF\otimes\widecheck{\cF}$.
Therefore, in any case we find the Fourier transform $\cG$ is compositely $2$-amiable. However, if $[+a]^*\cF\otimes\widecheck{\cF}$ is geometrically irreducible, then so is $\cG$ by Lemma \ref{lm:fouriertransform}, and $\cG$ should be $r^2$-amiable following the above arguments.
\endproof

\begin{remark}\label{rm:iterationwithB}
If the initial sheaf $\cF$ in Lemma \ref{lm:B-amiability} is of $SL_2$-type, i.e., the geometric monodromy group of $\cF$ is equal to $SL_2$, then $[+a]^*\cF\otimes\widecheck{\cF}$ is geometrically irreducible and of rank $2^2$ except for certain $a\in\bF_p,$ the number of which can be bounded in terms of $\fc(\cF).$ This is an immediate consequence of the the representation theory of $SU(2).$ 
Therefore, after applying $k$ times of $A$-processes, $\cF$ will become an irreducible sheaf of rank $2^{2^k}$ expect for certain special shifts, the number of which can also be bounded in terms of $\fc(\cF).$
\end{remark}

Although one cannot proceed with arbitrarily many iterations with $A$- and $B$-processes, Lemma \ref{lm:B-amiability} should be sufficient in many practical applications. However, on the other hand, one has to check the rank or amiability of the resultant sheaves after $A$- and $B$-processes. Thanks to the insight of Will Sawin, we will find that such checking works can be avoided if one starts from some particular sheaves, which are said to be {\it universally amiable}. More precisely, 
an admissible sheaf $\mathcal F_p$ on $\bP^1_{\bF_p}$ is universally amiable if it is a geometrically isotypic Fourier sheaf and its local monodromy at $\infty$ has all slopes $\leq 1$.  This will be given by Definition \ref{def:universallyamiable}. There are many typical examples of 
universally amiable sheaves arising in analytic number theory, and one may see  if $\cF$ is a universally amiable sheaf on $\bP_{\bF_p}^1$, then so are the sheaves after $A$- and $B$-processes up to some harmless errors (see Lemmas \ref{lm:ua-A-process} and \ref{lm:ua-B-process} for instance). In particular, all these sheaves will be compositely $\infty$-amiable. This will make the {\it arithmetic exponent pairs} for trace functions much easier to apply in practice. 
The details will be given in Appendix \ref{appendix:uasheaves}.

\smallskip

\section{$q$-analogue of the van der Corput method and exponent pairs}\label{sec:q-vdC}

\subsection{Framework of $q$-vdC}

Given a positive squarefree number $q$ and a composite trace function $K\bmod{q}$, we are interested in the cancellations among the average
\begin{align*}
\sum_{n\in I}K(n),
\end{align*}
where $I$ is some interval. If $|I|=q$, the sum is a complete one due to the periodicity of $K$.
If $|I|<q$, we are then working on an incomplete sum, a non-trivial bound of which is the main objective in many practical problems in analytic number theory. A common treatment is to transform the incomplete sum to complete ones (individual or on average) via Fourier analysis, and this can be demonstrated by Lemma \ref{lm:Bprocess} below.
In fact, we have
\begin{align*}
\sum_{n\in I}K(n)& \ll |I| q^{-1/2}\big(|\widehat{K}(0)|+|I|^{-1}q(\log q)\max_{h\neq0}|\widehat{K}(h)|\big)
\\\noalign{\vskip -2mm}
& \ll \|\widehat{K}\|_\infty\big(|I| q^{-1}+1\big)q^{1/2}\log q.\end{align*}
In many situations, we have $\|\widehat{K}\|_\infty := \max_x |\widehat{K}(x)|\ll q^{\varepsilon}$, and the above estimate is thus non-trivial for $|I|>q^{1/2+\varepsilon}$. This is what P\'olya and Vinogradov have done on estimates for incomplete multiplicative character sums, as we have mentioned in the first section. 
There seems no universal method to make a power-enlargement to the non-trivial range $|I|>q^{1/2+\varepsilon}$ for general $K$ and $q$; a relevant progress was recently made by Fouvry, Kowalski, Michel, {\it et al} \cite{FKM+17} to cover the gap between $q^{1/2+\varepsilon}$ and $q^{1/2}$. Nevertheless, Burgess \cite{Bu62, Bu63} succeeded when $K(n)=\chi(n)$ with general $q$, non-trivial $\chi$ and $|I|>q^{3/8+\varepsilon}$, which implies the first subconvexity for Dirichlet $L$-functions $L(\tfrac{1}{2}, \chi)$. A fascinating phenomenon was discovered by Heath-Brown \cite{HB78}, who was able to derive a Weyl-type bound for $L(\tfrac{1}{2}, \chi)$ if $q$ allows suitable factorizations. 
This is far earlier than his breakthrough on the greatest prime factors of $n^3+2$ in \cite{HB01} as discussed in Section \ref{sec:Introduction}.

As in the classical van der Corput method, 
estimates for such incomplete sums follow from the Weyl differencing and Poisson summation, 
which are usually called $A$-process and $B$-process, respectively. 
To formulate the two processes, we start with the general sums $S(\Psi; I)$ defined as in \eqref{eq:Psi-average}.

\begin{lemma}[$A$-process]\label{lm:Aprocess}
Assume $q=q_1q_2$ with $(q_1,q_2)=1$ and $\Psi_i:\bZ/q_i\bZ\rightarrow\bC$. 
Define $\Psi=\Psi_1\Psi_2$, then we have
\begin{align*}
|S(\Psi; I)|^2
\le 2\|\Psi_2\|_{\infty}^2 q_2 \Big(|I|+\sum_{0<|\ell|\leqslant L}\Big|\sum_{n\in\bZ} \mathbb{1}_{I}(n)\mathbb{1}_{I}(n+\ell q_2)\Delta_{\ell q_2}(\Psi_1)(n)\Big|\Big)
\end{align*}
for any $1\leqslant L\leqslant |I|/q_2,$
where $\mathbb{1}_{I}(\cdot)$ denotes the indicator function of $I$
and $\Delta_{\ell q_2}(\Psi_1)$ is defined as in \eqref{def:Deltahf}.
\end{lemma}

\begin{lemma}[$B$-process]\label{lm:Bprocess}
For $\Psi : \bZ/q\bZ\rightarrow\bC$, there exist an $a\in\bZ$ and some interval $\cI$ not containing $0$ with $|\cI|\leqslant q/|I|,$ such that
\begin{align*}
S(\Psi; I)&\ll \frac{|I|}{\sqrt{q}}\bigg(|\widehat{\Psi}(0)|+(\log q)\bigg|\sum_{h\in \cI}\widehat{\Psi}(h)
\mathrm{e}\Big(\frac{ha}{q}\Big)\bigg|\bigg).
\end{align*}
\end{lemma}

Lemmas \ref{lm:Aprocess} and \ref{lm:Bprocess} were stated explicitly by Irving \cite{Ir16} in a slightly different setting. 
Here we present the proof of Lemma \ref{lm:Aprocess} in our settings, which is essentially the same with Irving's.

\proof
Put $|I|=N$.
Assume $N>q_2,$ otherwise the lemma follows trivially. 
For any $\ell\in\bZ$, we have
\begin{align*}
S(\Psi; I) = \sum_{n\in\bZ} \mathbb{1}_{I}(n+\ell q_2)\Psi(n+\ell q_2).
\end{align*}
Summing over $\ell\leqslant L\in\bZ$ with $1\le L\leqslant N/q_2,$ we find
\begin{align*}
S(\Psi; I) = L^{-1} \sum_{\ell\leqslant L} \sum_{n\in\bZ} \mathbb{1}_{I}(n+\ell q_2) \Psi(n+\ell q_2).
\end{align*}

Note that $\Psi(n+\ell q_2)=\Psi_1(n+\ell q_2)\Psi_2(n)$.
It then follows that
\begin{align*}
|S(\Psi; I)|
& \le \|\Psi_2\|_{\infty} L^{-1} \sum_{n\in\bZ} \Big|\sum_{\ell\le L} \mathbb{1}_{I}(n+\ell q_2) \Psi_1(n+\ell q_2)\Big|.
\end{align*}
Since the outer sum over $n$ is of length at most $2N$, by Cauchy inequality we have
\begin{align*}
|S(\Psi; I)|^2
& \le 2\|\Psi_2\|_{\infty}^2 L^{-2} N\sum_{n\in\bZ}\Big|\sum_{\ell\le L} \mathbb{1}_{I}(n+\ell q_2)\Psi_1(n+\ell q_2)\Big|^2
\\
& \le 4\|\Psi_2\|_{\infty}^2L^{-1} N\Big(N +
\sum_{0<|\ell|\le L}\Big|\sum_{n\in\bZ} \mathbb{1}_{I}(n) \mathbb{1}_{I}(n+\ell q_2)\Delta_{\ell q_2}(\Psi_1)(n)\Big|\Big).
\end{align*}
This completes the proof of the lemma.
\endproof

In practice, the $A$-process is also known as the Weyl differencing and is usually employed with a number of iterations. The resultant sum is roughly of the same length with the original one, but the modulus becomes reasonably smaller, so that more cancellations become possible.
In contrast to $A$-process, the $B$-process transforms the original sum to a dual form (of different length but with the same modulus),
and this is better known as the Poisson summation (or completing method), going back to P\'olya and Vinogradov on the estimate for incomplete character sums.

Different combinations of $A$- and $B$-processes lead to different estimates for incomplete sums. 
We now recall some pioneer works that benefited from $q$-vdC.

\begin{itemize}
\item As mentioned before, Heath-Brown proved that $P^+(n^3+2)>n^{1+10^{-303}}$ for infinitely many $n$,
for which he used $\Psi(n)=\mathrm{e}(f_1(n)\overline{f_2(n)}/q)$ with $f_1,f_2\in\bZ[X]$ \cite{HB01}; the estimate for such exponential sums was recently used by Dartyge \cite{Da15} when studying the greatest prime factors of $n^4-n^2+1$ and by de la Bret\`eche \cite{dlB15} when extending Dartyge's result to any even unitary irreducible quartic polynomials with integral coefficients and with Galois group isomorphic to $\bZ/2\bZ\times\bZ/2\bZ$.

\item Earlier than the above example, Heath-Brown \cite{HB78} obtained a Weyl-type subconvexity for $L(\tfrac{1}{2}, \chi)$ when the modulus factorizes in a certain way.

\item  Graham and Ringrose \cite{GR90} got an extended zero-free region for Dirichlet $L$-functions with smooth moduli (on average), from which they deduced $\Omega$-results on least quadratic non-residues.

\item With the help of the ABC-conjecture, Heath-Brown \cite{HB10} was able to give a sharp estimate for the cubic Weyl sum $\sum_{n\leqslant N} \mathrm{e}(\alpha n^3)$
for any quadratic irrational $\alpha.$ 
The analytic exponential sum can, following a suitable Diophantine approximation to $\alpha$, be transformed to an algebraic one such that the modulus factorizes suitably.

\item Pierce \cite{Pi06} introduced the ideas of $q$-vdC to the square sieve of Heath-Brown, which enables her to derive the first non-trivial bound for the $3$-torsion of the class group of $\bQ(\sqrt{-D}).$ In their joint work on 
a conjecture of Serre concerning the number of rational points of
bounded height on a finite cover of projective space $\bP^{n-1}$, Heath-Brown and Pierce \cite{HBP12} can succeed in the special case of smooth cyclic covers of large degrees invoking the ideas of $q$-vdC to the power sieve.

\item The idea of $q$-vdC was contained implicitly in the recent breakthrough of Zhang \cite{Zh14} on bounded gaps between primes, and this was later highlighted by Polymath \cite{Po14} in the subsequent improvement.

\item Irving \cite{Ir15} has beaten the classical barrier of distribution of divisor functions in arithmetic progressions to smooth moduli, which previously follows from Weil's bound for Kloosterman sums to general moduli.

\item Irving \cite{Ir16} obtained a sub-Weyl bound for $L(\tfrac{1}{2}, \chi)$ when the modulus has only small prime factors.

\item Blomer and Mili\'cevi\'c \cite{BM15a} evaluated the second moment of twisted modular $L$-functions $L(\frac{1}{2},f\otimes\chi)$ as $\chi$ runs over primitive characters mod $q$ satisfying certain factorizations, where $f$ is a fixed (holomorphic or Maa{\ss}) Hecke cusp form.
\end{itemize}

The classical van der Corput method was also extended to algebraic exponential and character sums modulo prime powers, in which case one focuses on a fixed prime and the power tends to infinity. This is known as the {\it $p$-adic van der Corput method} and the reader is referred to \cite{BM15b} and \cite{Mi16} for more details.

\subsection{Arithmetic exponent pairs}
We now restrict our attention in Lemmas \ref{lm:Aprocess} and \ref{lm:Bprocess} to composite trace functions
and develop the method of (\textit{arithmetic}) \textit{exponent pairs} for incomplete sums of such trace functions, as an analogue of classical exponent pairs for analytic exponential sums initiated by Phillips (see e.g. \cite{GK91}).

Let $q$ be a positive squarefree number and assume $q$ has the suitable factorization $q=q_1q_2\cdots q_J$ for some $J\geqslant1$, where $q_j$'s are not necessarily primes but they are pairwise coprime. 
To each $q_j$, we associate a (possibly composite) trace function $K(\cdot,q_j).$
Put
\begin{align}\label{eq:K}
K(n)=\prod_{1\le j\le J} K(n,q_j).\end{align}
In what follows, we assume $K$ is admissible and, for each $p\mid q$,  
$\fc(\cF_p)\leqslant \fc$ for some uniform $\fc>0,$ where $\cF_p$ denotes the $\ell$-adic sheaf corresponding to $K_p.$ There is a convention that $K(\cdot,q)$ is identically 1 if $q=1$.

Let $\delta$ be a fixed positive integer such that $(\delta,q)=1$
and let $W_\delta:\bZ/\delta\bZ\rightarrow\bC$ be an arbitrary function, 
which we call {\it deformation factor} roughly.  

Keeping the above notation and assumptions, we consider the following sum
\begin{align}\label{eq:fS(K,W)}
\mathfrak{S}(K, W) := \sum_{n\in I} K(n) W_\delta(n),
\end{align}
where $I = ~]M, M+N]$ for some $M\in\bZ$. 
In what follows, we always assume $N<q\delta$, i.e., we will work on incomplete sums.

For $J\geqslant1$, put 
\begin{align}\label{eq:bold-kappalambdanumu}
\begin{cases}
\boldsymbol{\kappa}:= (\kappa_1,\kappa_2, \dots,\kappa_J)^t,\\
\boldsymbol{\lambda}:= (\lambda_1,\lambda_2, \dots,\lambda_J)^t,\\
\boldsymbol{\nu}:= (\nu_1,\nu_2, \dots,\nu_J)^t,
\end{cases}
\end{align}
where the superscript $t$ denotes the transpose of a vector. Let $(\boldsymbol{\kappa},\boldsymbol{\lambda},\boldsymbol{\nu})_J$ be a tuple such that
\begin{align}\tag*{$(\boldsymbol{\varOmega}_J):$}
\mathfrak{S}(K, W)
\ll_{J, \varepsilon, \fc} N^{\varepsilon}\|W_\delta\|_{\infty}\sum_{1\leqslant j\leqslant J} \Big(\frac{q_{J+1-j}}{N}\Big)^{\kappa_j}
N^{\lambda_j}\delta^{\nu_j},
\end{align}
where the implied constant is allowed to depend on $J,\varepsilon$ and $\fc$.
Let $(\kappa,\lambda,\nu)$ be a tuple such that
\begin{align}\tag*{$(\boldsymbol{\varOmega}):$}
\mathfrak{S}(K, W)
\ll_{\varepsilon, \fc} N^{\varepsilon}\|W_\delta\|_{\infty}(q/N)^{\kappa}N^{\lambda}\delta^{\nu},
\end{align}
where the implied constant is allowed to depend on $\varepsilon$ and $\fc.$ 
Ignoring the contributions from $\delta$, the terms $q_{J+1-j}/N$ and $q/N$ take the place
of derivatives of amplitude functions in classical analytic exponential sums.
The exponents $\boldsymbol{\nu}$ and $\nu$ are not quite essential in applications since we usually have extra summation of $\delta$ over some sparse sets, so that contributions from $\delta$ can be controlled effectively. 
For this reason, we keep the convention of {\it exponent pairs} although $(\boldsymbol{\kappa},\boldsymbol{\lambda},\boldsymbol{\nu})_J$ and $(\kappa,\lambda,\nu)$ appear as triads.
In practice, it is usually important to display the sizes of those exponents. To this end, we would like to introduce the restrictions
\begin{align}\label{eq:exponents-inequalies}
0\leqslant\kappa\leqslant\frac{1}{2}\leqslant\lambda\leqslant1,\ \ 0\leqslant\nu\leqslant1,
\end{align}
where subscripts can be added if necessary. For $J,L\geqslant1,$ denote by $\fA_q(J,L)$ the set of compositely $J$-amiable trace functions $K\bmod q$ such that $\widehat{K}\bmod q$ are compositely $L$-amiable.
We are now ready to introduce the following definitions.
\begin{definition}[Exponent pairs]\label{def:exponentpairs}
Suppose the coordinates in $(\boldsymbol{\kappa},\boldsymbol{\lambda},\boldsymbol{\nu})_J$ and $(\kappa,\lambda,\nu)$ satisfy the restrictions in $\eqref{eq:exponents-inequalies}$ with/without subscripts.

Let $J,L\geqslant1$ and $(q,\delta)=1,N\leqslant q\delta.$ We say
$(\boldsymbol{\kappa},\boldsymbol{\lambda},\boldsymbol{\nu})_J$ is an exponent pair of width $(J;L)$, if $(\boldsymbol{\varOmega}_J)$ holds for all $K\in\fA_q(J,L)$ and arbitrary functions $W_\delta:\bZ/\delta\bZ\rightarrow\bC$.

Similarly, we say
$(\kappa,\lambda,\nu)$ is an exponent pair of width $(J;L)$, if $(\boldsymbol{\varOmega})$ holds for all $K\in\fA_q(J,L)$ and arbitrary functions $W_\delta:\bZ/\delta\bZ\rightarrow\bC$. An exponent pair of width $(\infty;L)$ with some $L\geqslant1$ is called an arithmetic exponent pair.
\end{definition}

According to the above definition, one may see that an exponent pair $(\kappa,\lambda,\nu)$ of width $(J;L)$ is automatically of width $(J';L')$ with $J'\geqslant J,L'\geqslant L$, and in particular, such an exponent pair must be an arithmetic exponent pair. 

We now give some initial choices for the exponents in $A$- and $B$-processes.

\begin{proposition}\label{prop:initial}
Let $J\geqslant1.$
If $K$ is compositely $J$-amiable, then $(\boldsymbol{\varOmega}_1)$ holds with
\begin{align}\label{eq:A-initial}
(\boldsymbol{\kappa},\boldsymbol{\lambda},\boldsymbol{\nu})_1
= (\tfrac{1}{2}, \tfrac{1}{2}, 1),
\end{align}
and $(\boldsymbol{\varOmega})$ holds with
\begin{align}\label{eq:B-initial}
(\kappa, \lambda, \nu) = (\tfrac{1}{2}, \tfrac{1}{2}, 1).
\end{align}

In other words, $ (\tfrac{1}{2}, \tfrac{1}{2}, 1)$ is an exponent pair of width $(J;1)$ for each $J\geqslant1.$
\end{proposition}

\proof
By Lemma \ref{lm:Bprocess}, we have
\begin{align*}\mathfrak{S}(K, W)
\ll \frac{N}{\sqrt{q\delta}}\bigg(|\widehat{K}(0) \widehat{W}_\delta(0)|
+\log(q\delta)\bigg|\sum_{h\in \cI} \widehat{K}(h)\widehat{W}_\delta(h) \mathrm{e}\bigg(\frac{ah}{q\delta}\bigg)\bigg|\bigg)
\end{align*}
for some interval $\cI$ of length at most $q\delta/N$ and $a\in\bZ.$
Note that $\widehat{K}(h)\ll q^{\varepsilon}$
in view of the Chinese remainder theorem and Proposition \ref{prop:RH}. Hence we find
\begin{align*}
\mathfrak{S}(K, W)
\ll (qN)^{\varepsilon} \sqrt{q\delta}\|\widehat{W}_\delta\|_{\infty},
\end{align*}
which proves Proposition \ref{prop:initial} by noting that $\|\widehat{W}_\delta\|_{\infty}\leqslant \|W_\delta\|_{\infty}\sqrt{\delta}$.
\endproof

\begin{remark}
The classical van der Corput method for analytic exponential sums starts from the trivial 
exponent pair $(0,1)$, and this corresponds to the $q$-analogue
\[\mathfrak{S}(K, W)\ll  (q/N)^{0} N^{1} \delta^0.\]
However, we can always employ $B$-process to transform an incomplete sum
to a complete one, so that our initial exponent pair in Proposition \ref{prop:initial}
are in fact compared to $(\frac{1}{2},\frac{1}{2})=B\cdot(0,1)$ in the classical case.
\end{remark}

\begin{remark}
The introduction of $W_\delta$ makes the exponent pairs quite flexible in applications. For instance, to estimate a short exponential sum with respect to modulus $m$, which is not squarefree, one may take $q$ to be the squarefree part of $m$ and $\delta=m/q$, the squarefull part of $m$. By the Chinese remainder theorem, $W_\delta$ can be expressed in terms of certain exponentials and is thus absolutely bounded. Due to the sparse distribution of $\delta$, the contribution from $\delta$ usually is small and not harmful in practical applications. One can refer to the proof of Theorem \ref{thm:Brun-Titchmarsh} as a typical example.
\end{remark}

The main task to develop the method of (arithmetic) exponent pairs is to show how to
produce new exponent pairs from old ones. In the subsequent three sections,
we will present two alternative ways to produce new exponent pairs by virtue of 
$A$- and $B$-processes. In fact, there is no difference in the $B$-process and we proceed differently only in the $A$-process.
In the first approach, we assume $\delta$ is small in a certain way, such that the effect coming from the deformation factor $W_\delta$
can be eliminated by one $A$-process. As a result, the values of $\nu$ will take good shapes that are easy to control in the new exponent pairs.
In the second approach, we do not input the assumption on the size of $\delta$ and tend to prove the statements as general as possible.

For $\boldsymbol{\kappa},\boldsymbol{\lambda},\boldsymbol{\nu}$ given by \eqref{eq:bold-kappalambdanumu}, we define the $A_1,A_2$ maps via
\begin{align}\label{eq:exponentpair-afterA12}
A_i\cdot(\boldsymbol{\kappa},\boldsymbol{\lambda},\boldsymbol{\nu})_J:=(A_i\boldsymbol{\kappa},A_i\boldsymbol{\lambda},A_i\boldsymbol{\nu})_J
\end{align}
for $i=1,2$ with
\begin{align}\label{eq:exponentpair-afterA1}
\begin{cases}
A_1\boldsymbol{\kappa}=(\tfrac{1}{2},\tfrac{\kappa_1}{2}, \dots,\tfrac{\kappa_J}{2})^t\in\bR^{J+1},\\
A_1\boldsymbol{\lambda}=(1,\tfrac{\lambda_1+1}{2}, \dots,\tfrac{\lambda_J+1}{2})^t\in\bR^{J+1},\\
A_1\boldsymbol{\nu}=(1,0,\dots,0)^t\in\bR^{J+1}.
\end{cases}
\end{align}
and 
\begin{align}\label{eq:exponentpair-afterA2}
\begin{cases}
A_2\boldsymbol{\kappa}=(\tfrac{1}{2},\tfrac{\kappa_1}{2},\dots,\tfrac{\kappa_J}{2})^t\in\bR^{J+1},\\
A_2\boldsymbol{\lambda}=(1,\tfrac{\lambda_1+1}{2},\dots,\tfrac{\lambda_J+1}{2})^t\in\bR^{J+1},\\
A_2\boldsymbol{\nu}=(\frac{1}{2},\frac{\nu_1}{2},\dots,\frac{\nu_J}{2})^t\in\bR^{J+1}.\end{cases}
\end{align}
Note that $A_1$ and $A_2$ differ from each other in the last coordinate $\boldsymbol{\nu}$ only.
The above two maps correspond to two different approaches when applying the $A$-process. Regarding the $B$-process, we define
\begin{align}\label{eq:exponentpair-afterB}
B\cdot(\kappa,\lambda,\nu)
=\Big(\lambda-\frac{1}{2}, ~\kappa+\frac{1}{2}, ~\lambda+\nu-\kappa\Big).
\end{align}

In the following two sections, we will prove that the exponent pairs 
under the above maps can produce new exponent pairs in a certain way.
We conclude this section by proving that the inequalities in \eqref{eq:exponents-inequalies} (with or without subscripts)
are stable when applying $A$- and $B$-processes iteratively.
In fact, the stabilities under the maps $A_1$ and $A_2$ are clear and we only consider the case of map $B$. Clearly, we have
\begin{align*}
\lambda-\frac{1}{2}\leqslant\frac{1}{2}\leqslant\kappa+\frac{1}{2}\leqslant1\end{align*}
by assuming the first part in \eqref{eq:exponents-inequalies}. It remains to check
that 
\begin{align*}
\lambda+\nu-\kappa\leqslant1.
\end{align*}
Since 
one goes back the original estimate if $B$-process is applied twice and we may suppose $(\kappa,\lambda,\nu)$ comes from the $A$-process as in \eqref{eq:exponentpair-afterA1} and \eqref{eq:exponentpair-afterA2}, which we will examine right now. The case in \eqref{eq:exponentpair-afterA1} is easier and we only examine the case in \eqref{eq:exponentpair-afterA2}, for which we have (with subscripts)
\begin{align*}
\lambda_j+\nu_j-\kappa_j=\frac{\lambda_j+1}{2}+\frac{\nu_j}{2}-\frac{\kappa_j}{2}
&=\frac{\lambda_j+\nu_j-\kappa_j+1}{2}\leqslant \frac{1+1}{2}=1
\end{align*}
by induction.
We are done, i.e., starting from $(\frac{1}{2},\frac{1}{2},1)$, all new exponent pairs $(\boldsymbol{\kappa},\boldsymbol{\lambda},\boldsymbol{\nu})_J$ and $(\kappa,\lambda,\nu)$ produced by $A$- and $B$-processes must satisfy the inequalities in \eqref{eq:exponents-inequalies} with or without subscripts.

In what follows, we keep in mind that the inequalities \eqref{eq:exponents-inequalies} are always valid since we always start from Proposition \ref{prop:initial}, which will be assumed henceforth. It is an interesting and challenging problem to find other exponent pairs by inventing new approaches.

\smallskip

\section{Producing new exponent pairs: the first approach}\label{sec:exponentpairs-restricted}

Let us first state our theorems.

\begin{theorem}[$A$-process]\label{thm:exponentpair-A1}
Let $J\geqslant1$. 
If $(\boldsymbol{\kappa},\boldsymbol{\lambda},\boldsymbol{\nu})_J$ is an exponent pair of width $(J;1),$ then $A_1\cdot(\boldsymbol{\kappa},\boldsymbol{\lambda},\boldsymbol{\nu})_J$ as given by $\eqref{eq:exponentpair-afterA12}$ and $\eqref{eq:exponentpair-afterA1}$ is an exponent pair of width $(J+1;1)$.
\end{theorem}

\begin{theorem}[$B$-process]\label{thm:exponentpair-B}
If $(\kappa,\lambda,\nu)$ is an exponent pair of width $(1;1),$ then so is $B\cdot(\kappa,\lambda,\nu)$ as given by $\eqref{eq:exponentpair-afterB}.$
\end{theorem}

\begin{remark}
Let $J\geqslant1$.
Given a compositely $J$-amiable trace function $K\bmod q$, it is necessary to understand if the Fourier transform $\widehat{K}$ is also $J$-amiable. This is the case when $J=1$, and fails clearly when $J>1$. As an counterexample, one may take $\widehat{K}(x)=\ue(x^3/q)$, which is just $2$-amiable, however $K$ is compositely $\infty$-amiable according to \cite[Section 7.12]{Ka90}. On the other hand, the Fourier transform is involutive, and it is redundant to apply the $B$-process twice consecutively in practice. More precisely, one should apply the $B$-process at the first step, or after several applications of the $A$-process. In view of Lemmas \ref{lm:A-amiability} and \ref{lm:B-amiability}, we may see that if $(\kappa,\lambda,\nu)$ is an arithmetic exponent pair, then so is $B\cdot(\kappa,\lambda,\nu)$.
\end{remark}

The following proposition is an immediate consequence of Theorems \ref{thm:exponentpair-A1}, \ref{thm:exponentpair-B} and Proposition \ref{prop:initial}.

\begin{proposition}\label{prop:A-J>=1}
Let $J\geq2$.
Then 
$(\boldsymbol{\kappa},\boldsymbol{\lambda},\boldsymbol{\nu})_J$
with
\begin{align*}
\begin{cases}
\boldsymbol{\kappa}=(\tfrac{1}{2},\tfrac{1}{2},\dots,\tfrac{1}{2^{J-2}},\tfrac{1}{2^{J-1}})^t\in\bR^J,\\
\boldsymbol{\lambda}=(1,1,\dots,1,1-\frac{1}{2^{J-1}})^t\in\bR^J,\\
\boldsymbol{\nu}=(1,0,\dots,0,0)^t\in\bR^J
\end{cases}
\end{align*}
is an exponent pair of width $(J;1).$
\end{proposition}

\begin{remark}
As one can see from Theorem \ref{thm:exponentpair-A1} and Proposition \ref{prop:A-J>=1}, the saving against trivial estimates decays exponentially in $J$, so that one cannot make the savings ideally efficient
by taking larger value of $J$. 
In most applications, the choice $J=3$ or $4$ is usually sufficient. 
We will give explicit estimates for $\fS(K, W)$ in Section \ref{sec:explicitestimates} to display the role of each factor of the modulus.
\end{remark}

In certain applications, we may decompose $q$ suitably as a product of several squarefrees that are pairwise coprime, so that an estimate of the shape $(\boldsymbol{\varOmega})$  can be produced
by balancing all these terms in $(\boldsymbol{\varOmega}_J)$.
In such case, we would like to produce new exponent pairs from old ones 
by applying suitable combinations of $A$- or $B$-processes. 
As a counterpart of Theorem \ref{thm:exponentpair-B}, 
we should determine the shape of $(\kappa,\lambda,\nu)$ after one step of $A$-process. To this end, we define the maps
\begin{align}\label{eq:exponentpair-afterA1-optimatization}
A_1\cdot(\kappa,\lambda,\nu)
= \Big(\frac{\kappa}{2(\kappa+1)},~\frac{\kappa+\lambda+1}{2(\kappa+1)},~\frac{\kappa}{2(\kappa+1)}\Big),
\end{align}
\begin{align}\label{eq:exponentpair-afterA1-optimatization*}
A_1^*\cdot(\kappa,\lambda,\nu)
= \Big(\frac{\kappa}{2(\kappa+1)},~\frac{\kappa+\lambda+1}{2(\kappa+1)},~\frac{\kappa}{\kappa+1}\Big).
\end{align}

\begin{theorem}\label{thm:exponentpair-A1-optimization}
Suppose each prime factor of  $q$ is at most $q^\eta$ for any $\eta>0$. 
Let $(\kappa,\lambda,\nu)$ be an exponent pair of width $(J;1)$ for some $J\geqslant1$.
Then

{\rm (i)} $A_1\cdot (\kappa,\lambda,\nu)$ given by $\eqref{eq:exponentpair-afterA1-optimatization}$ is also an exponent pair of width $(J+1;1)$ if
\begin{align}\label{eq:restrictionforA1-treatment1}
q^{\kappa}N^{\lambda-\kappa}\geqslant\delta,\ \ \ (q\delta)^{2\kappa+1}\geqslant N^{3\kappa-\lambda+2}.
\end{align}

{\rm (ii)} $A_1^*\cdot (\kappa,\lambda,\nu)$ given by $\eqref{eq:exponentpair-afterA1-optimatization*}$ is also an exponent pair of width $(J+1;1)$ if
\begin{align}\label{eq:restrictionforA1-treatment2}
q^{\kappa}N^{\lambda-\kappa}\geqslant\delta^2.
\end{align}

\end{theorem}

Note that $\lambda\geqslant\frac{1}{2}$, from which we get
\begin{align*}
\frac{2\kappa+1}{3\kappa-\lambda+2}\geqslant\frac{2}{3}.
\end{align*}
Hence the constraint \eqref{eq:restrictionforA1-treatment2} is redundant if we assume that $N\leqslant (q\delta)^{\frac{2}{3}}$, which could be a reasonable assumption since we are usually interested in short averages.
Hence Theorem \ref{thm:exponentpair-A1-optimization} yields the following consequence by observing that
\begin{align*}
A_1^{-1}\cdot(\kappa,\lambda,*)
= \Big(\frac{2\kappa}{1-2\kappa},~\frac{2\lambda-1}{1-2\kappa},~*\Big),\quad 
B^{-1}\cdot (\kappa,\lambda,*)= \Big(\lambda-\frac{1}{2},~\kappa+\frac{1}{2},~*\Big).
\end{align*}
\begin{corollary}\label{corollary:exponentpair-A1-optimization}
Let $N\leqslant (q\delta)^{\frac{2}{3}}$. Suppose each prime factor of  $q$ is at most $q^\eta$ for any $\eta>0$.
Then $(\kappa,\lambda,\nu)$ is an exponent pair of width $(J;1)$ for some $J\geqslant1$, if there exists some other arithmetic exponent pair $(\kappa_0,\lambda_0,\nu_0)$ such that one of the following conditions holds:
\begin{itemize}
\item $(\kappa,\lambda,\nu)=A_1^*\cdot(\kappa_0,\lambda_0,\nu_0)$
and $q^{2\kappa}N^{2\lambda-2\kappa-1}\geqslant\delta^{1-2\kappa}$;
\item $(\kappa,\lambda,\nu)=BA_1^*\cdot(\kappa_0,\lambda_0,\nu_0)$
and $q^{2\lambda-1}N^{2\lambda+2\kappa+1}\geqslant\delta^{2-2\lambda}$.
\end{itemize}
\end{corollary}

\begin{remark}
One can see that Theorems \ref{thm:exponentpair-A1-optimization} and \ref{thm:exponentpair-B} coincide with classical exponent pairs for analytic exponential sums up to the exponents in $\delta$.

The modulus $q$ is required to satisfy quite good factorizations in Theorem \ref{thm:exponentpair-A1-optimization}.
We will look into this issue in Section \ref{sec:tauAPs-subconvexity}
on the distribution of divisor functions in arithmetic progressions and subconvexity of Dirichlet $L$-functions. Such an observation also plays an essential role 
in applications to the quadratic Brun--Titchmarsh theorem in Section \ref{sec:quadraticBT}.
\end{remark}

\begin{remark}
In applications of Theorems \ref{thm:exponentpair-B} and \ref{thm:exponentpair-A1-optimization} to producing new exponent pairs from old ones, one should determine if $K$ is compositely amiable enough in the sense of Definition \ref{def:amiablesheaf}. 
Of course, a sufficient condition is that $K$ is compositely $\infty$-amiable. 
However, in practice, that $K$ is compositely $J$-amiable for a certain large value of $J$ is usually sufficient 
since one becomes quite close to optimal exponent pairs after the first several iterations. 
For instance, the above arguments are applicable to the Weyl sum
$\sum_{n\leqslant N} \mathrm{e}(n^{2016}/q)$
if $q$ is smooth enough, although the summand is just $2015$-amiable.
\end{remark}

The following table gives the first several exponent pairs produced by different combinations of $A$- and $B$-processes 
to $(\frac{1}{2},\frac{1}{2},1)$ as shown in Theorems \ref{thm:exponentpair-A1-optimization} and \ref{thm:exponentpair-B}.
This can be compared with the table in \cite[p.117]{Ti86}.
\vskip -3mm
\begin{table}[htbp]
\renewcommand{\arraystretch}{1.4}
\renewcommand{\tabcolsep}{1.6mm}
\begin{tabular}{|c|c|c|c|c|}
\hline
Processes 
& 
$A_1$
& 
$A_1^2$
& 
$A_1^3$
& 
$BA_1^2$
\\
\hline
$(\kappa, \lambda,\nu)$
& 
$(\frac{1}{6}, \frac{2}{3}, \frac{1}{6})$ 
& 
$(\frac{1}{14}, \frac{11}{14}, \frac{1}{14})$ 
& 
$(\frac{1}{30}, \frac{13}{15}, \frac{1}{30})$ 
& 
$(\frac{2}{7}, \frac{4}{7}, \frac{11}{14})$ 
\\
\hline
\hline
Processes 
& 
$BA_1^3$
& 
$A_1BA_1^2$
& 
$A_1^2BA_1^2$
& 
$BA_1BA_1^2$
\\
\hline
$(\kappa, \lambda,\nu)$
&
$(\frac{11}{30}, \frac{8}{15}, \frac{13}{15})$ 
& 
$(\frac{1}{9}, \frac{13}{18}, \frac{1}{9})$ 
& 
$(\frac{1}{20}, \frac{33}{40}, \frac{1}{20})$ 
& 
$(\frac{2}{9}, \frac{11}{18}, \frac{13}{18})$ 
\\
\hline
\end{tabular}
\\~\\
\vskip 2mm
\caption{List of (arithmetic) exponent pairs. I}
\end{table}

\smallskip

\section{Proof of Theorems $\ref{thm:exponentpair-A1}$, \ref{thm:exponentpair-B} and \ref{thm:exponentpair-A1-optimization}} \label{sec:proofs-exponentpairs}

\subsection{Proof of Theorem $\ref{thm:exponentpair-A1}$}
In order to produce an exponent pair of width $(J+1;1)$, we now suppose $K$ is compositely $(J+1)$-amiable and $W_\delta$ does not vanish identically. Recall the definition of $\fS(K,W)$ in \eqref{eq:fS(K,W)} with $N=|I|.$

We assume
$N\geqslant q_{J+1}\delta$, otherwise the estimate for $\mathfrak{S}(K, W)$ follows trivially. We would like to make some initial treatments to 
$\fS(K, W)$, so that the $A$- and $B$-processes can be employed in a suitable way.
We first write
\begin{align*}
K(n)=\prod_{p\mid q}K_p(n)=\prod_{p\mid q}\Big(\sum_{j\in\cJ(p)}K_p(j;n)\Big),
\end{align*}
where $|\cJ(p)|\leqslant \fc(\cF_p)\leqslant \fc$ and $n\mapsto K_p(j;n)$ is geometrically isotypic
for each $j\in\cJ(p)$. Expanding the product over $p\mid q$, we may write
\begin{align*}
K(n)=\sum_{i\in\cJ}K_q(i;n),
\end{align*}
where $|\cJ|\leqslant \fc^{\omega(q)}$, and for each $i\in\cJ$, $K_q(i;n)$ is a composite trace function mod $q$ and its reduction mod $p$ is geometrically isotypic. These allow us to write
\begin{align}\label{eq:fS(K,W)-decomposition}
\fS(K, W)= \sum_{i\in\cJ}\fS(K_i, W),
\end{align}
where
\begin{align*}
\fS(K_i, W)
&=\sum_{n\in I}K_q(i;n)W_\delta(n).
\end{align*}
 
For each fixed $i\in\cJ$, we would like to apply the $A$-process to $\fS(K_i, W)$, getting
\begin{align}\label{eq:S(Kj,W)^2}
|\fS(K_i, W)|^2
\ll \|W_\delta\|_{\infty}^2\Big(L^{-1}N^2 + L^{-1}N \sum_{\ell\le L}|\mathfrak{A}(\ell)|\Big),
\end{align}
where $L$ is an integer to be chosen later with $1\leqslant L\leqslant N/(q_{J+1}\delta),$ and
\begin{align*}
\mathfrak{A}(\ell)
:= \sum_{n\in \widetilde{I}} K_i(n, Q_{J+1})\overline{K_i(n+\ell q_{J+1}\delta, Q_{J+1})}
\end{align*}
with some interval $\widetilde{I}\subseteq I$ and $Q_{J+1}=q/q_{J+1}$.
Here $K_i(n,Q_{J+1})$ denotes the reduction of $K_q(i;n)$ mod $Q_{J+1}$ as shown in the definition \eqref{eq:K}.

After one $A$-process, the expected estimate will follow from the bound for $\mathfrak{A}(\ell)$. Put 
\[q_j' := q_j/(\ell, q_j)
\qquad
(1\leqslant j\leqslant J),\]
in which case 
the trace function is a reduced form of $K(n,Q_{J+1})\overline{K(n+\ell q_{J+1},Q_{J+1})}$ 
mod $Q_{J+1}^*:=Q_{J+1}/(\ell, Q_{J+1}) = q_1'q_2'\cdots q_{J}'.$
Hence we may write
\begin{align*}
\mathfrak{A}(\ell)
= \sum_{n\in \widetilde{I}} \widetilde{K}(n) \widetilde{W}(n),
\end{align*}
where 
\begin{align*}\widetilde{K}(n) := K_i(n, Q_{J+1}^*)\overline{K_i(n+\ell q_{J+1}\delta, Q_{J+1}^*)},\end{align*}
and
\begin{align*}
\widetilde{W}(n)
:=|K_i(n, (\ell, Q_{J+1}))|^2
\end{align*}
is the new deformation factor $\bmod{(\ell, Q_{J+1})}$. 
According to Lemma \ref{lm:A-amiability}, we find the trace function $\widetilde{K}(n)$, which is well-defined mod $Q_J^*,$ can be expressed as a sum of at most $O(q^\varepsilon)$ compositely $J$-amiable trace functions mod $Q_J^*.$ To each of such compositely $J$-amiable trace functions, we would like to employ the estimates by assumption. 

{\bf Case I.} We first assume $N<Q_{J+1}^*(\ell, Q_{J+1})=Q_{J+1}.$

By assumption and $\|\widetilde{W}\|_\infty\ll q^\varepsilon$, we may deduce that
\begin{align}\label{eq:A(l)}
\mathfrak{A}(\ell)
\ll N^{\varepsilon}\sum_{1\leqslant j\leqslant J} \Big(\frac{q_{J+1-j}}{N}\Big)^{\kappa_j} 
N^{\lambda_j}(\ell, Q_{J+1})^{\nu_j}
\end{align}
for each $\ell\leqslant L$. Summing over $\ell\leqslant L$ and taking $L=[N/(q_{J+1}\delta)]$, it follows from \eqref{eq:S(Kj,W)^2} and \eqref{eq:A(l)} that
\begin{align*}
\fS(K, W)
& \ll \|W_\delta\|_\infty\Big\{\sqrt{Nq_{J+1}\delta}+N^{\varepsilon} 
\sum_{1\le j\le J} \Big(\frac{q_{J+1-j}}{N}\Big)^{\frac{1}{2}\kappa_j} N^{\frac{1}{2}(\lambda_j+1)}\Big\}.\end{align*}
Here we have used a fact that
\begin{align*}
\sum_{\ell\leqslant L}(\ell, Q_{J+1})^{\nu_j}\ll \tau(Q_{J+1})L,
\end{align*}
which is valid since we may assume $\nu_j\leqslant 1.$

{\bf Case II.} 
It remains to consider the case $N\geqslant Q_{J+1}$, which implies $q^{1/2}\leqslant (N q_{J+1})^{1/2}$.
An alternative estimate for $\fS(K, W)$ shows
\[\fS(K, W)\ll N^{\varepsilon} (q\delta)^{1/2} \|\widehat{W_\delta}\|_\infty\ll N^{\varepsilon} q^{1/2} \delta \|W_\delta\|_\infty\ll N^{\varepsilon}  (N q_{J+1})^{1/2}\delta \|W_\delta\|_\infty.\]

Combining the above two cases, we may arrive at the inequality
\begin{align*}
\fS(K, W)
& \ll \|W_\delta\|_\infty\Big\{\sqrt{Nq_{J+1}}\delta+N^{\varepsilon} 
\sum_{1\le j\le J} \Big(\frac{q_{J+1-j}}{N}\Big)^{\frac{1}{2}\kappa_j} N^{\frac{1}{2}(\lambda_j+1)} \delta^0\Big\},\end{align*}
which completes the proof.

\begin{remark}

Before closing this subsection, we now give some remarks on {\it universally amiable sheaves} (see Definition \ref{def:universallyamiable}). In practice, if each $K_p$ in the definition of $K\bmod q$ corresponds to some universally amiable sheaf, then from the above decomposition, and the reduction of each $K_q(i;n)$ mod $p$ is geometrically isotypic and universally amiable. After one $A$-process, it is natural to expect that the resultant function $\widetilde{K}$ can be 
expressed as a sum of at most $O(q^\varepsilon)$ composite trace functions, the reduction of which $\bmod p$ corresponds to some universally amiable sheaf on $\bP^1_{\bF_p}$. According to Lemma \ref{lm:ua-A-process}, this is indeed the case up to another harmless function, the support and sup-norm of which can be bounded in terms of the conductors of the underlying sheaves. Therefore, 
all the subsequent arguments for universally amiable sheaves will be similar to those for $(J+1)$-amiable sheaves.
\end{remark}

\subsection{Proof of Theorem \ref{thm:exponentpair-B}}

Suppose $K$ is compositely $1$-amiable, then so is $\widehat{K}$.
Firstly, we put $\Psi(n) = K(n) W_\delta(n)$ in Lemma \ref{lm:Bprocess}, 
so that the Fourier transform $\widehat{\Psi}$ of $\Psi$ over $\bZ/q\delta\bZ$ can be expressed by
$\widehat{\Psi}(h) = \widehat{K}(\overline{\delta}h) \widehat{W}_\delta(\overline{q}h)$,
where $q\overline{q}\equiv1\bmod\delta$ and $\delta\overline{\delta}\equiv1\bmod{q}.$ 
From Lemma \ref{lm:Bprocess}, we thus have
\begin{align*}
\mathfrak{S}(K, W)
\ll \frac{N}{\sqrt{q\delta}}
\bigg\{|\widehat{K}(0) \widehat{W}_\delta(0)|
+\log (q\delta) \bigg|\sum_{h\in \cI} \widehat{K}(\overline{\delta}h) \widehat{W}_\delta(\overline{q}h)
\mathrm{e}\Big(\frac{ha}{q\delta}\Big)\bigg|\bigg\}
\end{align*}
for some interval $\cI$ of length at most $q\delta/N$ and $a\in\bZ.$ 
By Proposition \ref{prop:RH}, we have 
\[\widehat{K}(0)=\prod_{p|q}\widehat{K}_p(0)\ll q^{\varepsilon}.\]
The above inequality reduces to
\begin{align}\label{eq:S(W,K)-B2}
\mathfrak{S}(K, W)
\ll \frac{N^{1+\varepsilon}}{\sqrt{q\delta}}
\bigg\{\sqrt{\delta}\|W_\delta\|_\infty+\bigg|\sum_{h\in \cI} \widehat{K}(\overline{\delta}h) \widehat{W}_\delta(\overline{q}h)
\mathrm{e}\Big(\frac{ha}{q\delta}\Big)\bigg|\bigg\}
\end{align}

Note that
\begin{align*}
\mathrm{e}\Big(\frac{ha}{q\delta}\Big)
= \mathrm{e}\Big(\frac{ha\overline{q}}{\delta}\Big)
\mathrm{e}\Big(\frac{ha\overline{\delta}}{q}\Big).
\end{align*}
We are now in a good position to apply the hypothesis for the trace function 
\begin{align*}
h\mapsto \widehat{K}(\overline{\delta}h) \mathrm{e}\Big(\frac{ha\overline{\delta}}{q}\Big)
\end{align*}
and the deformation factor
\begin{align}\label{eq:newdeformation}
h\mapsto \widehat{W}_\delta(\overline{q}h)
\mathrm{e}\Big(\frac{ha\overline{q}}{\delta}\Big).
\end{align}
Therefore, we derive from \eqref{eq:S(W,K)-B2} that
\begin{align*}
\mathfrak{S}(K, W)
& \ll N^{1+\varepsilon}q^{-\frac{1}{2}}\|W_\delta\|_{\infty}
+N^{1+\varepsilon}(q\delta)^{-\frac{1}{2}}
\Big(\frac{q}{q\delta/N}\Big)^\kappa (q\delta/N)^\lambda \delta^\nu \|\widehat{W}_\delta\|_{\infty}\\
& \ll N^{1+\varepsilon}q^{-\frac{1}{2}}\|W_\delta\|_{\infty}
+N^{\varepsilon}
(q/N)^{\lambda-\frac{1}{2}}N^{\kappa+\frac{1}{2}} \delta^{\lambda+\nu-\kappa} \|W_\delta\|_{\infty}.\end{align*}
Recall that $\kappa\leqslant\lambda$ and $N\leqslant q\delta$. Therefore, the second term will dominate and the above estimate can be reduced to
\begin{align*}
\mathfrak{S}(K, W)
&  \ll N^{\varepsilon}
(q/N)^{\lambda-\frac{1}{2}}N^{\kappa+\frac{1}{2}} \delta^{\lambda+\nu-\kappa} \|W_\delta\|_{\infty}
\end{align*}
as stated in Theorem \ref{thm:exponentpair-B}.

\subsection{Proof of Theorem \ref{thm:exponentpair-A1-optimization} and Corollary \ref{corollary:exponentpair-A1-optimization}}
We now assume $q=q_1q_2$ with $(q_1,q_2)=1$ and sizes of $q_1,q_2$ can be chosen freely due to small prime factors of $q$.

We first follow the arguments and convention in the proof of Theorem \ref{thm:exponentpair-A1}. In fact,
for each fixed $i\in\cJ$, we would like to apply the $A$-process to $K_q(i;n)$, getting
\begin{align*}
|\mathfrak{S}(K_i, W)|^2
\ll \|W_\delta\|_{\infty}^2\Big(L^{-1}N^2 + L^{-1}N \sum_{\ell\le L}|\mathfrak{B}(\ell)|\Big),
\end{align*}
where $L$ is an integer to be chosen later with $1\le L\leqslant N/(q_2\delta),$ and
\begin{align*}
\mathfrak{B}(\ell)
:= \sum_{n\in \widetilde{I}} K_i(n,q_1)\overline{K_i(n+\ell q_2\delta, q_1)}
\end{align*}
with some interval $\widetilde{I}\subseteq I$ and $|\widetilde{I}|\leqslant N$.

There are two different treatments to $\mathfrak{B}(\ell)$. The first one is to follow exactly the proof of Theorem \ref{thm:exponentpair-A1}, which yields
\begin{align}\label{eq:A-beforeoptimization-treatment1}
|\mathfrak{S}(K_i, W)|^2
\ll Nq_2\delta^2+N^{\varepsilon} (q_1/N)^{\kappa} N^{\lambda+1}.
\end{align}
The second treatment is to follow the proof of Theorem \ref{thm:exponentpair-A1} in the very beginning. The difference arises when distinguishing the two cases $|\widetilde{I}|<q_1$ and $|\widetilde{I}|\geqslant q_1.$ The latter case is related to the estimate for a complete sum, and we have
\begin{align}\label{eq:A-beforeoptimization-treatment2}
|\mathfrak{S}(K_i, W)|^2
\ll Nq_2\delta+N^{\varepsilon} (q_1/N)^{\kappa} N^{\lambda+1}+\frac{N^{2+\varepsilon}}{\sqrt{q_1}},
\end{align}
the last term of which comes from the case $|\widetilde{I}|\geqslant q_1.$

To balance the two terms in \eqref{eq:A-beforeoptimization-treatment1}, we may choose $q_1,q_2$ with
\begin{align*}
q_1\asymp q^{\frac{1}{\kappa+1}} N^{\frac{\kappa-\lambda}{\kappa+1}}\delta^{\frac{2}{\kappa+1}},
\qquad 
q_2\asymp q^{\frac{\kappa}{\kappa+1}}N^{\frac{\lambda-\kappa}{\kappa+1}}\delta^{-\frac{2}{\kappa+1}}.
\end{align*}
One always has $q_1\geqslant1$ since $\lambda-\kappa\leqslant1$ and $N\leqslant q\delta$. We also have $q_2\geqslant1$ thanks to \eqref{eq:restrictionforA1-treatment2}. Inserting \eqref{eq:A-beforeoptimization-treatment1} to \eqref{eq:fS(K,W)-decomposition}, we find
\begin{align*}
\mathfrak{S}(K, W)
\ll N^{\varepsilon} (q/N)^\frac{\kappa}{2(\kappa+1)} N^\frac{\kappa+\lambda+1}{2(\kappa+1)}\delta^\frac{\kappa}{\kappa+1}
\end{align*}
as expected.

To balance the first two terms in \eqref{eq:A-beforeoptimization-treatment2}, we may choose $q_1,q_2$ with
\begin{align*}
q_1\asymp q^{\frac{1}{\kappa+1}} N^{\frac{\kappa-\lambda}{\kappa+1}}\delta^{\frac{1}{\kappa+1}},
\qquad 
q_2\asymp q^{\frac{\kappa}{\kappa+1}}N^{\frac{\lambda-\kappa}{\kappa+1}}\delta^{-\frac{1}{\kappa+1}}.
\end{align*}
One always has $q_1\geqslant1$ since $\lambda-\kappa\leqslant1$ and $N\leqslant q\delta$. We also have $q_2\geqslant1$ thanks to the first restriction in \eqref{eq:restrictionforA1-treatment1}.
With the above choices of $q_1$ and $q_2$, one may check the third term in \eqref{eq:A-beforeoptimization-treatment2} is negligible in view of the second restriction in \eqref{eq:restrictionforA1-treatment1} and thus \eqref{eq:A-beforeoptimization-treatment2} yields
\begin{align*}
\mathfrak{S}(K, W)
\ll N^{\varepsilon} (q/N)^\frac{\kappa}{2(\kappa+1)} N^\frac{\kappa+\lambda+1}{2(\kappa+1)}\delta^\frac{\kappa}{2(\kappa+1)}
\end{align*}
as expected.

\smallskip

\section{Producing new exponent pairs: the second approach}\label{sec:exponentpairs-general}

Our second approach gives the following alternative iterations in the $A$-process: the difference only occurs in the coordinate $\boldsymbol{\nu}$ compared with the first approach.
\begin{theorem}[$A$-process]\label{thm:exponentpair-A2}
Let $J\geqslant1$. 
If $(\boldsymbol{\kappa},\boldsymbol{\lambda},\boldsymbol{\nu})_J$ is an exponent pair of width $(J;1),$ then $A_2\cdot(\boldsymbol{\kappa},\boldsymbol{\lambda},\boldsymbol{\nu})_J$ as given by $\eqref{eq:exponentpair-afterA12}$ and $\eqref{eq:exponentpair-afterA2}$ is an exponent pair of width $(J+1;1).$
\end{theorem}

As a counterpart of Proposition \ref{prop:A-J>=1}, we have following exponent pairs.
\begin{proposition}\label{prop:A-J>=1-general}
Let $J\geq2$.
Then 
$(\boldsymbol{\kappa},\boldsymbol{\lambda},\boldsymbol{\nu})_J$
with
\begin{align*}
\begin{cases}
\boldsymbol{\kappa}=(\tfrac{1}{2},\tfrac{1}{2},\dots,\tfrac{1}{2^{J-2}},\tfrac{1}{2^{J-1}})\in\bR^J,\\
\boldsymbol{\lambda}=(1,1,\dots,1,1-\frac{1}{2^{J-1}})\in\bR^J,\\
\boldsymbol{\nu}=(\frac{1}{2},\frac{1}{4},\dots,\frac{1}{4},\frac{1}{2})\in\bR^J
\end{cases}
\end{align*}
is an exponent pair of width $(J;1).$
\end{proposition}

We also have a counterpart of Theorem \ref{thm:exponentpair-A1-optimization}. To state our theorem, we need to introduce the $A_2$ map applied to $(\kappa,\lambda,\nu)$. Precisely, define
\begin{align}\label{eq:exponentpair-afterA2-optimization}
A_2\cdot(\kappa,\lambda,\nu)
= \bigg(\frac{\kappa}{2(\kappa+1)},~\frac{\kappa+\lambda+1}{2(\kappa+1)},~\frac{1}{2}\bigg).
\end{align}

\begin{theorem}\label{thm:exponentpair-A2-optimization}
Suppose each of prime factor of $q$ is at most $q^\eta$ for any $\eta>0$. 
If $(\kappa,\lambda,\nu)$ is an exponent pair of width $(J;1)$ for some $J\geqslant1$, then $A_2\cdot (\kappa,\lambda,\nu)$ is an exponent pair of width $(J+1;1).$
\end{theorem}

In view of Theorem \ref{thm:exponentpair-A2-optimization}, Lemmas \ref{lm:A-amiability} and \ref{lm:B-amiability} together with Remark \ref{rm:iterationwithB}, we have the following assertions.
\begin{theorem}\label{thm:exponentpair-afterAB}
Suppose each of prime factor of $q$ is at most $q^\eta$ for any $\eta>0$. Let $J\geqslant1$ be a positive integer and $(\kappa,\lambda,\nu)$ an exponent pair of width $(J;1).$ 

For each $k\geqslant0,$ $A_2^k\cdot (\kappa,\lambda,\nu)$ is
an exponent pair of width $(J+k;1),$ and $A_2^kB\cdot (\kappa,\lambda,\nu)$ is
an exponent pair of width $(k+3;1)$ if $2^{2k}\geqslant J.$ In particular, $A_2^k\cdot (\kappa,\lambda,\nu)$ and $A_2^kB\cdot (\kappa,\lambda,\nu)$ are both arithmetic exponent pairs.
\end{theorem}

\subsection{Sketch the proof of Theorem \ref{thm:exponentpair-A2}}
The proof of Theorem \ref{thm:exponentpair-A2} is quite similar to that of Theorem \ref{thm:exponentpair-A1}, and the essential difference is that the modulus $q_{J+1}$, instead of $q_{J+1}\delta$, is used here in the Weyl differencing ($A$-process). We now give the sketch of the proof and adopt the same notation as in Section \ref{sec:proofs-exponentpairs}, sometimes with different meanings.
By Lemma \ref{lm:Aprocess}, for each $\fS(K_i, W)$ in \eqref{eq:fS(K,W)-decomposition}, we find
\begin{align}\label{eq:S(W,K)^2-general}
|\mathfrak{S}(K_i, W)|^2
\ll L^{-1}N^2 + L^{-1}N \sum_{\ell\le L}|\mathfrak{A}(\ell)|,
\end{align}
where $L$ is an integer to be chosen later with $1\leqslant L \leqslant N/q_{J+1},$
\begin{align*}
\mathfrak{A}(\ell)
:= \sum_{n\in \widetilde{I}} K_i(n, Q_{J+1})\overline{K_i(n+\ell q_{J+1}, Q_{J+1})}
W_\delta(n)\overline{W_\delta(n+\ell q_{J+1})}
\end{align*}
with some interval $\widetilde{I}\subseteq I$ and $Q_{J+1}=q/q_{J+1}$.
Here $K(n,Q_{J+1})$ denotes the reduction of $K$ mod $Q_{J+1}$ as shown in the definition \eqref{eq:K}.

After one $A$-process, the expected estimate follows from the bound for $\mathfrak{A}(\ell)$. Put 
\[q_j' := q_j/(\ell, q_j)
\qquad
(1\leqslant j\leqslant J),\]
in which case 
the trace function is a reduced form of $K_i(n,Q_{J+1})\overline{K_i(n+\ell q_{J+1},Q_{J+1})}$ 
mod $Q_{J+1}^*:=Q_{J+1}/(\ell, Q_{J+1}) = q_1'q_2'\cdots q_{J}'.$
Hence we may write
\begin{align*}
\mathfrak{A}(\ell)
= \sum_{n\in \widetilde{I}} \widetilde{K}(n) \widetilde{W}(n),
\end{align*}
where $\widetilde{K}(n) := K_i(n, Q_{J+1}^*)\overline{K_i(n+\ell q_{J+1}, Q_{J+1}^*)}$,
and
\begin{align*}
\widetilde{W}(n)
:= W_\delta(n)\overline{W_\delta(n+\ell q_{J+1})}|K_i(n, (\ell, Q_{J+1}))|^2
\end{align*}
is the new deformation factor $\bmod{\delta(\ell, Q_{J+1})}$. 
According to Lemma \ref{lm:A-amiability}, we find the trace function $\widetilde{K}(n)$, which is well-defined mod $Q_J^*,$ can also be expressed as a sum of at most $O(q^\varepsilon)$ compositely $J$-amiable trace functions mod $Q_J^*.$

{\bf Case I.} We first assume $N<Q_{J+1}^*\delta(\ell, Q_{J+1})=Q_{J+1}\delta.$

By hypothesis, we may deduce that
\begin{align}\label{eq:A(l)-general}
\mathfrak{A}(\ell)
\ll N^{\varepsilon}\sum_{1\leqslant j\leqslant J} \Big(\frac{q_{J+1-j}}{N}\Big)^{\kappa_j} 
N^{\lambda_j}\delta^{\nu_j}(\ell, Q_{J+1})^{\nu_j}\|\widetilde{W}\|_\infty
\end{align}
for each $\ell\leqslant L$.
Summing over $\ell\leqslant L$ and taking $L=[N/q_{J+1}]$, it follows from \eqref{eq:S(W,K)^2-general} and \eqref{eq:A(l)-general} that
\begin{align*}
\fS(K, W)
& \ll N^\varepsilon\|W_\delta\|_\infty\Big\{(Nq_{J+1})^{1/2}+
\sum_{1\le j\le J} \Big(\frac{q_{J+1-j}}{N}\Big)^{\frac{1}{2}\kappa_j} N^{\frac{1}{2}(\lambda_j+1)} \delta^{\frac{\nu_j}{2}}\Big\}.
\end{align*}

{\bf Case II.} 
It remains to consider the case $N\geqslant Q_{J+1}\delta$.
An alternative estimate for $\mathfrak{S}(K, W)$ reads
\[\fS(K, W)\ll N^{\varepsilon} q^{1/2}\delta \|W_\delta\|_\infty\ll N^{\varepsilon}  (N q_{J+1}\delta)^{1/2} \|W_\delta\|_\infty\]
as given by \eqref{eq:B-initial} and also
$q\delta\leqslant N q_{J+1}$
by assumption.

Combining the above two cases, we conclude the following bound
\begin{align*}
\fS(K, W)
& \ll N^\varepsilon\|W_\delta\|_\infty\Big\{(Nq_{J+1}\delta)^{1/2}+
\sum_{1\le j\le J} \Big(\frac{q_{J+1-j}}{N}\Big)^{\frac{1}{2}\kappa_j} N^{\frac{1}{2}(\lambda_j+1)} \delta^{\frac{\nu_j}{2}}\Big\},
\end{align*}
which completes the proof of Theorem \ref{thm:exponentpair-A2}.

\subsection{Proof of Theorem \ref{thm:exponentpair-A2-optimization}}
We also assume $q=q_1q_2$ with $(q_1,q_2)=1$ and the sizes of $q_1,q_2$ can be chosen freely.

Following the similar arguments to those in the proof of Theorem \ref{thm:exponentpair-A1}, we get
\begin{align*}
\mathfrak{S}(K, W)^2
& \ll N^{\varepsilon}\|W_\delta\|_\infty^2(Nq_2\delta+(q_1/N)^{\kappa} N^{\lambda+1} \delta^{\nu})
\end{align*}
by assumption. To balance the two terms, we may choose $q_1,q_2$ by
\begin{align}\label{eq:q1q2-general}
q_1\asymp q^{\frac{1}{\kappa+1}} N^{\frac{\kappa-\lambda}{\kappa+1}},
\qquad 
q_2\asymp q^{\frac{\kappa}{\kappa+1}}N^{\frac{\lambda-\kappa}{\kappa+1}},
\end{align}
so that the above estimate becomes
\begin{align*}
\mathfrak{S}(K, W)
\ll N^{\varepsilon} (q/N)^\frac{\kappa}{2(\kappa+1)} N^\frac{\kappa+\lambda+1}{2(\kappa+1)}
\delta^{\frac{1}{2}}  \|W_\delta\|_\infty
\end{align*}
as expected, completing the proof of Theorem \ref{thm:exponentpair-A2-optimization}.

\smallskip

\section{Arithmetic exponent pairs with constraints}\label{sec:exponentpairs-constraints}

In practice, the assumption about good factorizations in Theorems \ref{thm:exponentpair-A1-optimization} and \ref{thm:exponentpair-A2-optimization} is a bit strong; this is usually not satisfied in certain applications. More precisely, the sizes of some factors of $q$ might be restricted, so that one cannot make balances as freely as in \eqref{eq:q1q2-general}.
Suppose $q$ has a factor which is of size $Q$, and the complementary divisor has no prime factors exceeding $q^\eta$ for any $\eta>0$. We can prove a variant of Theorems \ref{thm:exponentpair-A1-optimization} and \ref{thm:exponentpair-A2-optimization} subject to this constraint. 

In fact, we can argue as in the proof of Theorem \ref{thm:exponentpair-A2-optimization}. 
Making the choice \eqref{eq:q1q2-general}, one should assume $q_1\geqslant Q$, 
in which case the iteration can be applied.
Note that
\begin{align*}
A_2^{-1}\cdot (\kappa,\lambda,*)
= \bigg(\frac{2\kappa}{1-2\kappa},~\frac{2\lambda-1}{1-2\kappa},~*\bigg)
\end{align*}
for $(\kappa,\lambda,\nu)\neq(\frac{1}{2},\frac{1}{2},1)$.
Hence we obtain an exponent pair $(\kappa,\lambda,\nu)$, 
if it is produced by $A$-processes from $(\frac{1}{2},\frac{1}{2},1)$ and 
$q^{1-2\kappa}N^{2\kappa-2\lambda+1}\geqslant Q$.
On the other hand, if the exponent pair $(\kappa,\lambda,\nu)$ is produced by $A$-processes with an extra $B$-process in the beginning, one should replace $(\kappa,\lambda)$ by $(\lambda-\frac{1}{2},\kappa+\frac{1}{2})$ in the above restriction, thus getting the new constraint $q^{2-2\lambda}N^{2\lambda-2\kappa-1}\geqslant Q$. We now summarize our argument as follows.

\begin{theorem}\label{thm:exponentpair-constraints}
Suppose $q$ has a divisor of size $Q$, and the complementary divisor of $q$ has only prime factors at most $q^\eta$ for any $\eta>0$.
Then $(\kappa,\lambda,\nu)$ is an arithmetic exponent pair, provided that one of the following conditions holds$:$
\par
{\rm (a)} $(\kappa,\lambda,\nu)=A_2^k\cdot(\frac{1}{2},\frac{1}{2},1)$ for some $k\geqslant1$ and
\begin{align}\label{eq:exponentpair-A-constraints}
q^{1-2\kappa}N^{2\kappa-2\lambda+1}\geqslant Q.\end{align}
\par
{\rm (b)} 
$(\kappa,\lambda,\nu)=BA_2^k\cdot(\frac{1}{2},\frac{1}{2},1)$ for some $k\geqslant1$ and
 \begin{align}\label{eq:exponentpair-B-constraints}
q^{2-2\lambda}N^{2\lambda-2\kappa-1}\geqslant Q.\end{align}
\end{theorem}

One can also obtain relevant restrictions for other types of $(\kappa,\lambda,\nu)$.

\begin{remark}
In practical applications, Theorem \ref{thm:exponentpair-constraints} can serve as an alternative to Theorems \ref{thm:exponentpair-A1-optimization} and \ref{thm:exponentpair-B}. More precisely, one can switch the role of $\delta$ in different situations. For instance, if the size of $\delta$ is not too big, one can utilize 
Theorem \ref{thm:exponentpair-A1-optimization} or Corollary \ref{corollary:exponentpair-A1-optimization}. 
However,
if $\delta$ is reasonably large, it is not a good choice to shift by $q_{J+1}\delta$ as 
in the proof of Theorem \ref{thm:exponentpair-A1-optimization}, and it is better to leave $\delta$
in its place and one has to display $W_\delta$ as the deformation factor. Of course,
if $W_\delta$
satisfies certain oscillations, we can pick out the squarefree part $\delta^\flat$ of $\delta$ and also the reduction of $W_\delta\bmod{\delta^\flat}$ to form the new 
trace functions, in which case Theorem \ref{thm:exponentpair-constraints} applies since part of the modulus does not allow sufficiently good factorizations.
One can find this observation in the later proof of Theorem \ref{thm:Brun-Titchmarsh}. 
\end{remark}

\smallskip

\section{Explicit estimates for sums of trace functions}\label{sec:explicitestimates}
In this section, we give some explicit estimates for averages of algebraic trace functions.
In particular, we display the dependence of the upper bound on divisors of $q$.
In what follows, we assume $q=q_1q_2\cdots q_J$ with $J\geqslant2.$

The first estimate can be compared with Proposition
\ref{prop:A-J>=1}. 
\begin{theorem}[$A^JB$-estimate]\label{thm:S(W,K)-A^kB}
Assume $K$ is compositely $J$-amiable with $J\geqslant2$. For $|I|\ll (q\delta)^{O(1)}$, we have
\begin{align*}
\mathfrak{S}(K, W)
&\ll_{J,\varepsilon,\fc} |I|^{1+\varepsilon}
\bigg\{\frac{\|\widehat{W}_\delta\|_{\infty}}{\sqrt{q\delta}}\omega_0+\|\widehat{W}_\delta\|_{\infty}\sqrt{\frac{q_J\delta}{|I|}}+\|W_\delta\|_{\infty}\sum_{j=2}^{J-1}\Big(\frac{q_{J+1-j}}{|I|}\Big)^{2^{-j}}\\
&\ \ \ \ \ \ \ \ \ \ \ \ \ +\|W_\delta\|_{\infty}\Big(\frac{q_1}{|I|^2}\Big)^{2^{-J}}\bigg\},
\end{align*}
where $\omega_0=1$ if $|I|\geqslant q\delta$ and vanishes otherwise. 
\end{theorem}

\begin{proof}[Sketch of the proof]
The case $\omega_0=0$ follows from Proposition
\ref{prop:A-J>=1} immediately. In the case $\omega_0=1$, we may split the summation by periodicity: There are $[|I|/q\delta]$ complete intervals of length $q\delta$, each of which can contribute a complete exponential sum mod $q\delta$. The remaining part is of length $|I|-q\delta[|I|/q\delta]$, which would be equal to $|I|$ if $\omega_0=0.$ Each above complete sum is equal to
\begin{align*}
\sum_{a\bmod{q\delta}} K(a) W_\delta(a)
= \Big\{\prod_{p\mid q}\sum_{a\bmod p}K_p(a)\Big\} \Big\{\sum_{a\bmod\delta} W_\delta(a)\Big\}
\ll (q\delta)^{1/2+\varepsilon} \|\widehat{W}_\delta\|_{\infty}
\end{align*}
by Riemann hypothesis for $K_p$ (Proposition \ref{prop:RH}).
The total contribution from all complete sums is then
\begin{align*}
\ll [|I|/(q\delta)] (q\delta)^{1/2+\varepsilon} \|\widehat{W}_\delta\|_{\infty}
\ll |I| (q\delta)^{-1/2+\varepsilon} \|\widehat{W}_\delta\|_{\infty}
\end{align*}
as expected.
\end{proof}

\begin{remark}
The case $\delta=1$ in Theorem \ref{thm:S(W,K)-A^kB} was claimed in \cite[Section 6]{Po14} without proof. 
Heath-Brown \cite{HB01} obtained, in the case $K(n)=\mathrm{e}(f_1(n)\overline{f_2(n)}/q)$ with $f_1,f_2\in\bZ[X],$ a similar estimate with an extra term roughly of the shape $|I|^{1+\varepsilon} q_1^{2^{-J-1}} \|\widehat{W}_\delta\|_{\infty}$.
\end{remark}

We have an alternative estimate following Theorem \ref{thm:exponentpair-A2}. 
\begin{theorem}[$A^JB$-estimate]\label{thm:S(W,K)-A^kB-general}
Assume $K$ is compositely $J$-amiable with $J\geqslant2$. For $|I|\ll (q\delta)^{O(1)}$, we have
\begin{align*}
\mathfrak{S}(K, W)
&\ll_{J,\varepsilon,\fc} |I|^{1+\varepsilon}
\bigg\{\frac{\|\widehat{W}_\delta\|_{\infty}}{\sqrt{q\delta}}\omega_0+\|\widehat{W}_\delta\|_{\infty}\sqrt{\frac{q_J}{|I|}}+\|W_\delta\|_{\infty}\delta^{\frac{1}{4}}\sum_{j=2}^{J-1}\Big(\frac{q_{J+1-j}}{|I|}\Big)^{2^{-j}}\\
&\ \ \ \ \ \ \ \ \ \ \ \ \ \ +\|W_\delta\|_{\infty}\sqrt{\delta}\Big(\frac{q_1}{|I|^2}\Big)^{2^{-J}}\bigg\},
\end{align*}
where $\omega_0=1$ if $|I|\geqslant q\delta$ and vanishes otherwise.
\end{theorem}

Some alternative estimates could be obtained by making some initial transformations before applying the $A^JB$-estimate. 
\begin{theorem}[$BA^JB$-estimate]\label{thm:S(W,K)-BA^kB}
Assume the Fourier transform $\widehat{K}$ is compositely $J$-amiable with $J\geqslant2$. For $|I|\ll (q\delta)^{O(1)}$, we have
\begin{align*}
\mathfrak{S}(K, W)
&\ll_{J,\varepsilon,\fc}\|\widehat{W}_\delta\|_{\infty}(q\delta)^{\frac{1}{2}+\varepsilon}
\bigg\{\frac{|I|}{q\delta}\omega_0+\sqrt{\frac{|I| q_J}{q}}+\sum_{j=2}^{J-1}\Big(\frac{|I| q_{J+1-j}}{q\delta}\Big)^{2^{-j}}+\Big(\frac{|I|^2q_1}{(q\delta)^2}\Big)^{2^{-J}}\bigg\},
\end{align*}
where $\omega_0=1$ if $|I|\geqslant q\delta$ and vanishes otherwise.
\end{theorem}
\begin{remark}
 Irving \cite{Ir16} has beaten the Weyl bound for Dirichlet $L$-functions using estimates for character sums that are of the same strength with Theorem \ref{thm:S(W,K)-BA^kB} with $K(n)=\chi(n)\overline{\chi}(n+h),~h\neq0$. 
\end{remark}

\smallskip

\section{Proof of Theorem \ref{thm:Brun-Titchmarsh}: Quadratic Brun--Titchmarsh theorem}\label{sec:quadraticBT}

We complete the proof of Theorem \ref{thm:Brun-Titchmarsh} in this section. 

\subsection{Linear sieves}
We introduce some conventions and fundamental results on the linear Rosser--Iwaniec sieve.
Let $\cA=(a_n)$ be a finite sequence of integers and $\cP$ a set of prime numbers. Define the sifting function
\begin{align*}S(\cA,\cP,z)=\sum_{(n,P(z))=1}a_n
\quad\text{with}\quad
P(z)=\prod_{p<z, \, p\in\cP}p.
\end{align*}
For squarefree $d$ with all its prime factors belonging to $\cP$, 
we consider the subsequence $\cA_d=(a_n)_{n\equiv 0\!\bmod{d}}$ and 
the congruence sum
\begin{align*}A_d=\sum_{n\equiv 0\!\bmod{d}}a_n.\end{align*}
We assume that $\cA_d$ is well located in $\cA$ in the following sense: 
There are an appropriate approximation $\mathfrak{M}$ to $A_1$ and 
a multiplicative function $g$ supported on squarefree numbers with all its prime factors belonging to $\cP$ verifying
\[0<g(p)<1\quad(p\in\cP)\]
such that
\par
(a) the remainder 
\[r(\cA,d)=A_d-g(d)\mathfrak{M}\]
is small on average over $d\mid P(z)$;

(b) there exists a constant $L>1$ such that
\begin{align*}
\frac{V(z_1)}{V(z_2)}
\leqslant\frac{\log z_2}{\log z_1}\bigg(1+\frac{L}{\log z_1}\bigg)
\quad\text{with}\quad
V(z)=\prod_{p<z, \, p\in\cP}(1-g(p))
\end{align*}
for $2\leqslant z_1< z_2$.

Let $F$ and $f$ be the continuous solutions to the system
\begin{align}\label{eq:sieve-F}
\begin{cases}
sF(s) =2\mathrm{e}^\gamma & \text{for $\,0< s\leqslant 2$},
\\
sf(s)=0                                    & \text{for $\,0<s\le 2$},
\\
(sF(s))'=f(s-1)                         & \text{for $\,s>2$},
\\
(sf(s))'=F(s-1)                         & \text{for $\,s>2$},
\end{cases}
\end{align}
where $\gamma$ is the Euler constant.

For $k\ge 1$, denote by $\tau_k(n)$ the number of ways of expressing $n$ as
the product of $k$ positive integers. An arithmetic function $\lambda(d)$ is of {\it level}
$D$ and {\it order} $k$, if
\begin{align*}
\lambda(d)=0
\quad(d>D)
\qquad\text{and}\qquad
|\lambda(d)|\leqslant\tau_k(d)
\quad
(d\leqslant D).
\end{align*}
Let $r\geqslant2$ be a positive integer. We say that $\lambda$ is {\it well-factorable of degree $r$}, if for every decomposition $D = D_1D_2\cdots D_r$ with $D_1, D_2, \dots, D_r\geqslant1,$ there exist $r$ arithmetic functions 
$\lambda_1, \lambda_2, \dots, \lambda_r$ such that
\begin{align*}
\lambda=\lambda_1*\lambda_2*\cdots *\lambda_r\end{align*}
with each $\lambda_j$ of level $D_j$ and order $k$.

We now state the following fundamental result of Iwaniec \cite{Iw80}.

\begin{lemma}\label{lm:Iwaniec}
Let $r\ge 2$ and $0<\varepsilon<\tfrac{1}{8}$.
Under the above hypothesis, we have
\begin{align*}
S(\cA,\cP,z)
\leqslant \mathfrak{M}V(z)\bigg\{F\bigg(\frac{\log D}{\log z}\bigg)+E\bigg\}
+\sum_{t\leqslant T}\sum_{d\mid P(z)}\lambda_t(d) r(\cA,d)
\end{align*}
for all $z\geqslant 2$,
where $F$ is given by the system \eqref{eq:sieve-F}, 
$T$ depends only on $\varepsilon$, 
$\lambda_t(d)$ is well-factorable of level $D,$ degree $r$ and order $1$, 
and $E\ll\varepsilon +\varepsilon^{-8} \mathrm{e}^L(\log D)^{-1/3}.$
\end{lemma}

Moreover, in Iwaniec's construction of $\lambda_t(d)$ one may see that $d$ has only small prime factors, and you may decompose $d$ into the product of certain divisors with suitable sizes. 
As mentioned before, such well-factorization of $d$, or equivalently the well-factorization of $\lambda_t(d)$,
is essential in estimates for exponential sums arising in the proof of Theorem \ref{thm:Brun-Titchmarsh}. The original statement in \cite[Theorem 4]{Iw80} is a bit different, and it implies Lemma \ref{lm:Iwaniec} by combining an iterative application of Lemma 1 therein. 

\subsection{Quadratic congruences}
Before starting the proof of Theorem \ref{thm:Brun-Titchmarsh}, we recall one classical result of Gau{\ss} on the representation of numbers by binary
quadratic forms. We refer to {\it Disquisitiones Arithmeticae} or Smith's
report \cite{Sm65} for a very clear description of this theory in a form suitable for our
purpose.

\begin{lemma}\label{lm:Gausslemma}
 Let $\ell\geqslant 1.$ If
\begin{align}\label{eq:cong}
a^2+1\equiv0 \bmod{\ell}
\end{align}
is solvable for $a\bmod{\ell}$, then $\ell$ can be represented properly as a sum of two squares
\begin{align}\label{eq:form}
\ell = r^2+s^2,\quad (r, s)=1, \quad r>0, \quad s>0.
\end{align}
There is a one to one correspondence between the incongruent solutions $a\bmod{\ell}$ to \eqref{eq:cong} 
and the solutions $(r, s)$ to \eqref{eq:form} given by
\[\frac{a}{\ell}=\frac{\overline{s}}{r}-\frac{s}{r(r^2+s^2)}.\]

For each integer $d\ge 1,$ we put $d=d_1d_2$ with $d_2=(d,r^{\infty})$. Then
\begin{align}\label{eq:da/l}
\frac{\overline{d}a}{\ell}&\equiv -\frac{r\overline{d_2(r^2+s^2)}}{d_1s}+ \frac{r}{ds(r^2+s^2)}- \frac{r\overline{d_1s(r^2+s^2)}}{d_2}\; \bmod{1}.
\end{align}
In particular, if $d$ is squarefree, we have
\begin{align*}
\frac{\overline{d}a}{\ell}&\equiv -\frac{r\overline{d_2(r^2+s^2)}}{d_1s}+ \frac{r}{ds(r^2+s^2)}\; \bmod{1}.
\end{align*}
All the mod inverses are well-defined with respect to relevant denominators without special precision.
\end{lemma}

\proof 
It suffices to prove \eqref{eq:da/l} and the remaining part can be actually found in \cite{Sm65} or \cite[Lemma 2]{DI82b}.
Note that $d_2=(d,r^{\infty})$, thus $(d_1,r\ell)=(d_2,d_1s\ell)=1.$
It follows that
\begin{align*}
\frac{\overline{d}a}{\ell}
& \equiv d_2^* \overline{d_1}\frac{a}{\ell} \; \bmod{1},\end{align*}
where $d_2^*d_2\equiv\overline{d_1}d_1\equiv1\bmod{\ell}$.
From the Chinese remainder theorem, we may choose $d_2^*,\overline{d_1}$ such that
$d_2^*d_2\equiv1\bmod{d_1s\ell}$ and $\overline{d_1}d_1\equiv1\bmod{r\ell}$.
On the other hand, we have
\begin{align*}
\frac{a}{\ell}\equiv\frac{\overline{s}(r^2+s^2)}{r(r^2+s^2)}-\frac{s}{r(r^2+s^2)} \; \bmod{1},
\end{align*}
where $s\overline{s}\equiv1\bmod{(r(r^2+s^2)}.$ 
Combining these yields
$$
\frac{\overline{d}a}{\ell}
\equiv d_2^*\bigg(\frac{\overline{d_1s}(r^2+s^2)}{r(r^2+s^2)}-\frac{s\overline{d_1}}{r(r^2+s^2)}\bigg) \; \bmod{1}.
$$
Noting that
\[\frac{\overline{v}}{u}+\frac{\overline{u}}{v}\equiv\frac{1}{uv}\;\bmod{1},\ \ (u,v)=1\]
with tuples $(u, v) = (d_1s, r(r^2+s^2))$ and $(u, v)=(d_1, r(r^2+s^2))$, we get
\begin{equation}\label{eq:congruence}
\begin{aligned}
\frac{\overline{d}a}{\ell}
& \equiv  - \frac{d_2^*\overline{r}}{d_1s}
+ \frac{d_2^*s \overline{r(r^2+s^2)}}{d_1}
+ \frac{d_2^*(r^2+s^2)}{d_1rs(r^2+s^2)}
- \frac{d_2^*s}{d_1r(r^2+s^2)} 
\; \bmod{1}.
\end{aligned}
\end{equation}

Furthermore, we have
\begin{align*}
- \frac{d_2^*\overline{r}}{d_1s}
+ \frac{d_2^*s \overline{r(r^2+s^2)}}{d_1}
& \equiv - \frac{\overline{d_2r}}{d_1s}
+ \frac{\overline{d_2r}s^2 \overline{(r^2+s^2)}}{d_1s} \; \bmod{1}
\\
& \equiv - \frac{r\overline{d_2(r^2+s^2)}}{d_1s} \; \bmod{1}
\end{align*}
and
\begin{align*}
\frac{d_2^*(r^2+s^2)}{d_1rs(r^2+s^2)}
- \frac{d_2^*s}{d_1r(r^2+s^2)} 
& = \frac{d_2^*r}{d_1s(r^2+s^2)}
\\
& \equiv r\bigg(\frac{1}{d_1d_2s(r^2+s^2)} - \frac{\overline{d_1s(r^2+s^2)}}{d_2}\bigg)
\; \bmod{1},
\end{align*}
which can be simplified to 
$$ \equiv \frac{r}{ds(r^2+s^2)}
\; \bmod{1}$$
if $d$ is squarefree, since $d_2\mid r$ in that case.
Inserting these to \eqref{eq:congruence}, we are done.
\endproof

\subsection{Initial treatment by linear sieves}

We now start the proof of Theorem \ref{thm:Brun-Titchmarsh}.
Let $Y=X^{1-\varepsilon}$ and we introduce a  smooth function $\fJ$ which is supported on $[Y,X]$, satisfying
\begin{align*}
\begin{cases}
\fJ(x)=1\ \ &\text{for } x\in[Y+1,X-1],\\
\fJ(x)\geqslant0\ \ &\text{for } x\in\bR,\\
\fJ^{(j)}\ll_j Y^{-j} &\text{for all } j\geqslant0.
\end{cases}
\end{align*}
Therefore,
\begin{align}\label{eq:Q_l(Y,X)}
Q_\ell(X)=Q_\ell(Y,X)+O(X^{\varepsilon/2}Y/\ell),
\end{align}
where
\[Q_\ell(Y,X)=\sum_{p^2+1\equiv0\bmod \ell}\fJ(p).\]
The $O$-term in \eqref{eq:Q_l(Y,X)} is acceptable since $Y=X^{1-\varepsilon}$ and all the following treatments are operated over $Q_\ell(Y,X).$

The initial step is to transform sums over primes to those over integers via linear sieves. To do so, we consider the congruence sum
\begin{align*}
A_d(Y,X; \ell) 
:= \sum_{\substack{n\geqslant1\\n^2+1\equiv0\!\bmod{\ell}\\n\equiv 0\!\bmod{d}}} \fJ(n)
= \sum_{\substack{n\geqslant1\\d^2n^2+1\equiv0\!\bmod{\ell}}} \fJ(dn).
\end{align*}

To transform the sum over $n$, we would like to appeal to the following Poisson summation formula.
\begin{lemma}\label{lm:Poisson}
Let $g$ be a smooth function with compact support on $\bR.$ 
For $X\geqslant1$ and $q\geqslant1,$ we have
\begin{align*}
\sum_{n\equiv a\!\bmod{q}}g\Big(\frac{n}{X}\Big)
= \frac{X}{q} \sum_{h\in\bZ} \widehat{g}\Big(\frac{hX}{q}\Big) \mathrm{e}\Big(\frac{ha}{q}\Big).
\end{align*}
\end{lemma}
From Lemma \ref{lm:Poisson} it follows that
\begin{align*}
A_d(Y,X; \ell) 
= \frac{1}{d\ell} \sum_{h} \widehat{\fJ}\Big(\frac{h}{d\ell}\Big)
\sum_{\substack{a\bmod{\ell}\\a^2+1\equiv0\!\bmod{\ell}}} \mathrm{e}\Big(\frac{a\overline{d}h}{\ell}\Big),
\end{align*}
where we have used the implicit condition $(d, \ell)=1$. 
The zero-th frequency is expected to produce the main contribution, i.e., $\widehat{\fJ}(0)\varrho(\ell)(d\ell)^{-1}$. 
Define
\begin{align*}
r_d(Y,X; \ell)
& := A_d(Y,X; \ell)-\widehat{\fJ}(0)\frac{\varrho(\ell)}{d\ell}= \frac{1}{d\ell}\sum_{h\neq0} \widehat{\fJ}\Big(\frac{h}{d\ell}\Big)
\sum_{\substack{a\bmod{\ell}\\a^2+1\equiv0\!\bmod{\ell}}} \mathrm{e}\Big(\frac{a\overline{d}h}{\ell}\Big).
\end{align*}

By the M\"obius formula, we have
\begin{align}\label{eq:Q_l(X)-sieve}
Q_{\ell}(Y,X) 
= \sum_{n^2+1\equiv0\!\bmod{\ell}} \fJ(n) \sum_{d\mid (n, P(z))} \mu(d)
\qquad
\big(z=\sqrt{X}\big).
\end{align}
Let $(\lambda_d)$ be a linear upper bound sieve of level $D$, so that $1*\mu\leqslant1*\lambda$. 
Then 
\begin{align*}
Q_{\ell}(Y,X)
\leqslant \sum_{\substack{d\leqslant D\\d\mid P(z)}} \lambda_d A_d(Y,X; \ell)
\leqslant \widehat{\fJ}(0)F\bigg(\frac{\log D}{\log z}\bigg)
+ \sum_{\substack{d\leqslant D\\d\mid P(z)}} \lambda_d r_d(Y, X; \ell),
\end{align*}
where $F$ is defined by \eqref{eq:sieve-F} and
\begin{align*}
V(z) 
:= \prod_{p<z, \, p\nmid \ell} \bigg(1-\frac{1}{p}\bigg)
= \frac{\ell}{\varphi(\ell)} \cdot \frac{\mathrm{e}^{-\gamma}}{\log z}\bigg\{1+O\bigg(\frac{1}{\log X}\bigg)\bigg\}.
\end{align*}
Thanks to Lemma \ref{lm:Iwaniec}, we may choose well-factorable remainder terms in the above application of linear Rosser--Iwaniec sieve. We thus conclude the following result by noting that $\widehat{\fJ}(0)=X+O(X^{1-\varepsilon}).$

\begin{proposition}\label{prop:Q_l(X)bound}
Let $\varepsilon>0$ and $D<X.$ For any given $J\geqslant 2,$ we have
\begin{align*}
Q_{\ell}(X)
\leqslant \{2+O(\varepsilon)\}\frac{\varrho(\ell)}{\varphi(\ell)}\frac{X}{\log D}+\sum_{t\leqslant T(\varepsilon)}\sum_{\substack{d\leqslant D\\d\mid P(z)}}\lambda_t(d) r_d(Y,X; \ell),\end{align*}
where $T(\varepsilon)$ depends only on $\varepsilon,$ $\lambda_t(d)$ is well-factorable of level $D,$ degree $J$ and order $1.$
\end{proposition}

Given a modulus $\ell$ of certain size, we hope the level $D$ could be chosen as large as possible. 
Theorem \ref{thm:Brun-Titchmarsh} is then implied by the following key proposition.

\begin{proposition}\label{prop:Q_l(X)bound-average}
Let $J$ be a sufficiently large integer and let $\lambda$ be well-factorable of degree $J.$
With the same notation as Proposition $\ref{prop:Q_l(X)bound}$,
for any $\varepsilon>0$ and  $(D, L) := (X^{\gamma(\theta)-\varepsilon}, X^\theta),$ there exists some $\delta=\delta(\varepsilon)>0$ such that
\begin{align*}
\sum_{\ell\sim L}\Big|\sum_{d\leqslant D} \mu(d)^2 \lambda(d) r_d(Y, X; \ell)\Big|\ll X^{1-\delta},\end{align*}
where $\gamma(\theta)$ is given by \eqref{eq:gamma(theta)}
and the implied constant depends on $\varepsilon$ and $J$. 
\end{proposition}

Proposition \ref{prop:Q_l(X)bound-average} implies, in the ranges \eqref{eq:gamma(theta)}, that
\begin{align*}
\sum_{t\leqslant T(\varepsilon)} \sum_{\substack{d\leqslant D\\d\mid P(z)}} \lambda_t(d) r_d(Y, X; \ell)
\ll \ell^{-1}X^{1-\delta}
\end{align*}
save for at most $O(L(\log L)^{-A})$ exceptional values of $\ell\in[L,2L]$. 
This, together with Proposition \ref{prop:Q_l(X)bound}, yields
\begin{align*}
Q_{\ell}(X)\leqslant\{2+o(1)\}\frac{\varrho(\ell)}{\varphi(\ell)}\frac{X}{\log D}\end{align*}
for such $\ell$. It now suffices to prove Propositions \ref{prop:Q_l(X)bound-average} 
in order to derive Theorem \ref{thm:Brun-Titchmarsh}.

\subsection{Reducing to exponential sums}
Recall that
\begin{align*}
r_d(Y, X; \ell)
= \frac{1}{d\ell} \sum_{\substack{a\bmod{\ell}\\a^2+1\equiv0\!\bmod{\ell}}} \sum_{h\neq0}
\widehat{\fJ}\Big(\frac{h}{d\ell}\Big) \mathrm{e}\Big(\frac{a\overline{d}h}{\ell}\Big).
\end{align*}
From the rapid decay of $\widehat{\fJ}$, we get
\begin{align*}r_d(Y, X; \ell)&=\frac{1}{d\ell}\sum_{\substack{a\bmod{\ell}\\a^2+1\equiv0\!\bmod{\ell}}}\sum_{0<|h|\leqslant H}\widehat{\fJ}\Big(\frac{h}{d\ell}\Big)\mathrm{e}\Big(\frac{a\overline{d}h}{\ell}\Big)+O(X^{-100})\end{align*}
for $H=DLX^{-1+\varepsilon}$. Upon the choices in \eqref{eq:gamma(theta)}, we remark that $\gamma(\theta)+\theta>1$ for $\theta\in[\frac{1}{2},\frac{16}{17}[$, so that $H\gg X^\varepsilon$ in this situation.
By dyadic device, it then suffices to prove that
\begin{align}\label{eq:errorterm-L1estimate}
\sum_{\ell\sim L} \sum_{\substack{a\bmod{\ell}\\a^2+1\equiv0\!\bmod{\ell}}}
\bigg|\sum_{d\sim  D} \mu(d)^2 \lambda(d) \sum_{h\leqslant H}
\widehat{\fJ}\Big(\frac{h}{d\ell}\Big)\mathrm{e}\Big(\frac{a\overline{d}h}{\ell}\Big)\bigg|
\ll DLX^{1-\varepsilon'}
\end{align}
for some $\varepsilon'>0$ (depending on $\varepsilon$).
Due to the well-factorization of $\lambda$, we may write 
\[\lambda=\alpha*\beta
\quad\text{with}\quad
\alpha=\alpha_1*\cdots*\alpha_{J_1}
\quad\text{and}\quad
\beta=\beta_1*\cdots*\beta_{J_2},\]
where $\alpha_j$ and $\beta_j$ are of level $M_j$ and $N_j$, respectively, and 
$$M_1\cdots M_{J_1}=M,
\quad
N_1\cdots N_{J_2}=N,
\quad 
MN=D 
\quad\text{and}\quad
J_1+J_2=J.
$$
Hence the left-hand side of \eqref{eq:errorterm-L1estimate} is at most
\begin{align*}
&\ll X^\varepsilon\sum_{m\sim M} |\alpha(m)| \sum_{\ell\sim L}  \sum_{\substack{a\bmod{\ell}\\a^2+1\equiv0\!\bmod{\ell}}} \!\!\!\!
\bigg|\sum_{n\sim N} \beta(n) \sum_{h\leqslant H}
\widehat{\fJ}\Big(\frac{h}{\ell mn}\Big) \mathrm{e}\Big(\frac{ah\overline{mn}}{\ell}\Big)\bigg|
\end{align*}
for some choices $M,N$ and $\alpha,\beta$ as above.
By Cauchy inequality, this is further bounded by
\begin{align*}
&\ll X^\varepsilon(ML)^{\frac{1}{2}}\cB(M,N)^{\frac{1}{2}}
\end{align*}
with
\begin{align*}
\cB(M, N) 
:= \sum_{m\sim M} |\alpha(m)| \sum_{\ell\in\bZ} W\Big(\frac{\ell}{L}\Big)
\!\!\!\!\!\! \sum_{\substack{a\bmod{\ell}\\a^2+1\equiv0\!\bmod{\ell}}} \!\!\!\!
\bigg|\sum_{n\sim N} \beta(n) \sum_{h\leqslant H}
\widehat{\fJ}\Big(\frac{h}{\ell mn}\Big) \mathrm{e}\Big(\frac{ah\overline{mn}}{\ell}\Big)\bigg|^2,
\end{align*}
where $W$ is a non-negative smooth function with compact support in $[\tfrac{1}{2}, \tfrac{5}{2}]$ 
and takes value $1$ in $[1, 2]$.
We would like to show, for some small $\delta>0$, that
\begin{align}\label{eq:B(M,N)bound}
\cB(M, N)\ll (ML)^{-1}(DL)^2X^{2-\delta}
\end{align}
holds
in the ranges of \eqref{eq:gamma(theta)}. We here leave a remark that $\mu^2(mn)=1$ is always kept in mind, which yields $(m,n)=1$ directly and will not be displayed to ease the presentation.

Squaring out and switching summations, we get
\begin{align*}
\cB(M, N)
& = \mathop{\sum\sum}_{n_1,n_2\sim N} \beta(n_1) \overline{\beta(n_2)}
\mathop{\sum\sum}_{h_1,h_2\leqslant H}\sum_{m\sim M}|\alpha(m)|
\sum_{\ell\in\bZ} \Phi_{\ell}(\bh; m, \bn) 
\varrho_{\overline{m}(h_1\overline{n_1}-h_2\overline{n_2})}(\ell),
\end{align*}
where $\bh=(h_1, h_2)$, $\bn=(n_1, n_2)$,
$\varrho_n(\ell)$ is defined as in \eqref{eq:rho(n,l)} and
\begin{align*}
\Phi_{\ell}(\bh; m, \bn)
:= W\Big(\frac{\ell}{L}\Big) \widehat{\fJ}\Big(\frac{h_1}{\ell mn_1}\Big)
\overline{\widehat{\fJ}\Big(\frac{h_2}{\ell mn_2}\Big)}.
\end{align*}

According to Lemma \ref{lm:Gausslemma}, for each $a\bmod{\ell}$ with $a^2+1\equiv0 \bmod{\ell}$, 
we have
\begin{align*}
\frac{ah_j\overline{mn_j}}{\ell}
\equiv -\frac{h_jr\overline{a_j(r^2+s^2)}}{smn_j/a_j}+ \frac{h_jr}{mn_js(r^2+s^2)}\;\bmod{1}
\end{align*}
for $j=1, 2$, where $a_j=(mn_j,r),$ $r>0$, $s>0$, $(r^2+s^2,mn_1n_2)=(r,s)=1$ and $r^2+s^2=\ell$. Therefore, the exponential sum $ \varrho_{\overline{m}(h_1\overline{n_1}-h_2\overline{n_2})}(\ell)
$ can be rewritten as
\begin{align*}
\sum_{\substack{a_1\mid mn_1\\ a_2\mid mn_2}} 
\mathop{\sum\sum}_{\substack{r, \, s>0, \, (r,s)=1\\ r^2+s^2=\ell\\ (mn_j, r)=a_j,j=1,2\\ (mn_1n_2, r^2+s^2)=1}} \!\!\! \mathrm{e}\bigg(\!\frac{h_2r\overline{a_2(r^2+s^2)}}{smn_2/a_2}
-\frac{h_1r\overline{a_1(r^2+s^2)}}{smn_1/a_1}\bigg)
\mathrm{e}\bigg(\frac{r(h_1/n_1-h_2/n_2)}{ms(r^2+s^2)}\bigg).
\end{align*}
In view of the Taylor expansion
\begin{align*}
\mathrm{e}\bigg(\frac{r(h_1/n_1-h_2/n_2)}{ms(r^2+s^2)}\bigg)
= 1 + O\bigg(\frac{Hr}{LMNs}\bigg),
\end{align*}
it follows that
\begin{equation}\label{eq:B(M,N)}
\begin{aligned}
\cB(M, N)
& = \sum_{\substack{n_1\sim N\\ n_2\sim N}} \beta(n_1) \overline{\beta(n_2)}
\sum_{\substack{h_1\le H\\ h_2\le H}} \sum_{m\sim M} |\alpha(m)|
\sum_{\substack{a_1\mid mn_1\\a_2\mid mn_2}}
 \mathop{\sum\sum}_{\substack{r, \, s>0, \, (r,s)=1\\(mn_j, r)=a_j,j=1,2\\ (mn_1n_2, r^2+s^2)=1}} 
\\
& \hskip -5mm\times
\Phi_{r^2+s^2}(\bh; m, \bn)
\mathrm{e}\bigg(\!\frac{h_2r\overline{a_2(r^2+s^2)}}{smn_2/a_2}
- \frac{h_1r\overline{a_1(r^2+s^2)}}{smn_1/a_1}\bigg)+O(H^3N).
\end{aligned}
\end{equation}

The innermost sums over $r,s$ can be split into dyadic intervals, $r\sim R,s\sim S$, say, where $R^2+S^2\asymp L$ and so that $R,S\ll\sqrt{L}.$ The following arguments will focus on estimating the sum over $r$ effectively to gain cancellations and the other sums will be treated trivially.
We thus consider the following exponential sum
\begin{align*}
T(\ba ; m, \bn, s) := \sideset{}{^*}\sum_{\substack{r\sim R\\(mn_j, r)=a_j,j=1,2}} \mathrm{e}(\xi(r)),
\end{align*}
where $\ba=(a_1, a_2),\bn=(n_1,n_2)$, 
the symbol $*$ means that $(r, s)=1$ and $(mn_1n_2,r^2+s^2)=1$, and
\begin{align}\label{eq:xi(r)}
\xi(r)\equiv\frac{h_2r\overline{a_2(r^2+s^2)}}{smn_2/a_2}-\frac{h_1r\overline{a_1(r^2+s^2)}}{smn_1/a_1} \; \bmod{1}.
\end{align}

Thanks to the well-factorizations of $m,n_1,n_2$, we may appeal to $q$-vdC and the method of 
arithmetic exponent pairs developed in Section \ref{sec:q-vdC} to give sharp estimates for 
$T(\ba, m, \bn, s)$. Note that the restrictions $(mn_1, r)=a_1$ and 
$(mn_2, r)=a_2$ can be rewritten equivalently as $r\equiv 0 \bmod{[a_1, a_2]}$ with some extra coprime conditions on $r$. There are many ways to relax the congruence condition; for instance, one can write $r=r'[a_1, a_2]$ with $r'\sim R/[a_1, a_2]$, or the orthogonality of additive characters can be ultilized to detect the divisibility. For the economy of presentation, we appeal to the first treatment and write
\begin{align*}
T(\ba; m, \bn, s) 
= \sideset{}{^*}\sum_{r\sim R_0} 
\mathrm{e}(\xi([a_1,a_2]r)),
\end{align*}
where $R_0=R/[a_1,a_2].$

\subsection{Estimates for exponential sums}
Let $j=1,2$. For $a_j=(mn_j,r)$, we may write $a_j = b c_j$ in a unique way, 
where $b=(m,r), c_j=(n_j,r).$ Note that $(bc_j,s)=1.$ 
Put $\tm:=m/b$, so that one can rewrite \eqref{eq:xi(r)} as
\begin{align*}
\xi(r)\equiv\frac{h_2r\overline{a_2(r^2+s^2)}}{\tm(n_2/c_2)s}
-\frac{h_1r\overline{a_1(r^2+s^2)}}{\tm(n_1/c_1)s}
\;\bmod{1}.
\end{align*}
Write $\tn=[n_1/c_1, n_2/c_2]$, and $\tn=(n_1/c_1)\sigma_1=(n_2/c_2)\sigma_2$, so that $(\sigma_1,\sigma_2)=1$. 
Thus we may have
\begin{align*}
\xi(r)\equiv{r}\Delta\cdot\frac{\overline{a_1a_2(r^2+s^2)}}{\tm\tn s}\;\bmod{1},
\ \ \ \ 
\Delta:=h_2\sigma_2a_1-h_1\sigma_1a_2.
\end{align*}
Denote by $(\tm\tn s)^\flat$ and $(\tm\tn s)^\sharp$ the squarefree and squarefull parts of $\tm\tn s$, respectively. 
We then have
\begin{align*}
\xi(r)
& \equiv r\Delta\bigg(\frac{\overline{a_1a_2(\tm\tn s)^\sharp(r^2+s^2)}}{(\tm\tn s)^\flat}+\frac{\overline{a_1a_2(\tm\tn s)^\flat(r^2+s^2)}}{(\tm\tn s)^\sharp}\bigg)\;\bmod{1}
\\
&\equiv \frac{r\Delta}{k}\frac{\overline{a_1a_2(\tm\tn s)^\sharp(r^2+s^2)}}{(\tm\tn s)^\flat/k}+r\Delta \frac{\overline{a_1a_2(\tm\tn s)^\flat(r^2+s^2)}}{(\tm\tn s)^\sharp}\;\bmod{1},\end{align*}
where $k:= (\Delta, (\tm\tn s)^\flat)$.

We would like to apply the method of arithmetic exponent pairs 
in the case of $q=(\tm\tn s)^\flat/k,$ $\delta=(\tm\tn s)^\sharp$, the trace function
\begin{align}\label{eq:K-q}
K: x\mapsto 
\mathrm{e}\bigg(\frac{x\Delta[a_1,a_2]}{k}\frac{\overline{a_1a_2(\tm\tn s)^\sharp([a_1,a_2]^2x^2+s^2)}}{(\tm\tn s)^\flat/k}\bigg)
\end{align}
and the deformation factor
\begin{align}\label{eq:W-delta}
W_\delta: x\mapsto 
\mathrm{e}\bigg(x\Delta[a_1,a_2] \frac{\overline{a_1a_2(\tm\tn s)^\flat([a_1,a_2]^2x^2+s^2)}}{(\tm\tn s)^\sharp}\bigg).
\end{align}
Note that $\|W_\delta\|_\infty=1.$

Before stating the precise estimate for $T(\ba; m, \bn, s)$, 
one has to produce an admissible upper bound for $\widehat{W}_\delta$. 
In Appendix B, we will present an {\it almost square-root cancellation} for complete algebraic exponential sums, see Theorem \ref{thm:S(lambda,c)bound} therein, from which we may conclude, 
if $W_\delta$ is chosen as \eqref{eq:W-delta}, that
\begin{align*}
\|\widehat{W}_\delta\|_{\infty}
\leqslant 2\cdot6^{\omega(\delta)}(h_2\sigma_2a_1-h_1\sigma_1a_2,(\tm\tn s)^\ddagger)
\cdot \Xi(\tm\tn s)^{\frac{1}{2}},
\end{align*}
where the superscript $\ddagger$ yields 
\begin{align}\label{eq:definition-ddagger}
n^\ddagger:=\prod_{p^2\| n}p,
\end{align}
$\Xi(\cdot)$ is defined as \eqref{eq:Xi}.

In order to introduce exponent pairs, one has to examine the amiability of the trace function in \eqref{eq:K-q} firstly. To this end, we just need to consider the following function
\begin{align*}
K_p(x)=\ue\Big(\frac{h_1x\overline{(x^2+h_2)}}{p}\Big)
\end{align*}
with $p\nmid h_1.$ First,  $x\mapsto K_p(x)$ is an $\infty$-amiable trace function of rank 1, which is thus geometrically irreducible. The property on rank 1 is very important, so that the resultant trace function is always of rank 1 and $\infty$-amiable after arbitrarily many $A$-processes. In this way, the original trace function will be transformed to some other function of the shape
\begin{align*}
K_p^*(x)=\ue\Big(\frac{g_1(x)\overline{g_2(x)}}{p}\Big),
\end{align*}
where $g_1,g_2\in\bF_p[X]$ are two polynomials with $\deg(g_1)<\deg(g_2)$ over $\bF_p.$
Following the arguments in Lemma \ref{lm:B-amiability}, the Fourier transform of $K_p^*$ is compositely $\infty$-amiable.

Let $(\kappa,\lambda,\nu)$ be an exponent pair of width $(J_0;L_0)$ for some $J_0,L_0\geqslant1$. Note that $\tm,\tn$ have well factorizations due to the sieve weight in Lemma \ref{lm:Iwaniec}.
Hence we have
\begin{align*}
T(\ba; m, \bn, s)
& \ll R^{\varepsilon}\Big(\frac{q}{R_0}\Big)^{\kappa}{R_0}^\lambda\delta^\nu\ll X^{\varepsilon}\bigg(\frac{(\tm\tn s)^\flat/k}{R_0}\bigg)^{\kappa} {R_0}^\lambda \big((\tm\tn s)^\sharp\big)^\nu
\end{align*}
with a relevant restriction, coming from Theorem \ref{thm:exponentpair-constraints} according to the type of $(\kappa,\lambda,\nu)$. This estimate is valid for incomplete exponential sums, i.e., $R<q\delta[a_1,a_2]=\tm\tn s[a_1, a_2]/k$, and otherwise, we appeal to the completing method only, getting
\begin{align*}
T(\ba; m, \bn, s)
& \ll R^{\varepsilon}\bigg[\frac{R_0}{q\delta}\bigg]\sqrt{q\delta} \, \|\widehat{W}_\delta\|_\infty\\
&\ll X^\varepsilon (h_2\sigma_2a_1-h_1\sigma_1a_2, (\tm\tn s)^\ddagger)\cdot R_0\sqrt{\frac{\Xi(\tm\tn s) k}{\tm\tn s}}.
\end{align*}
Inserting the above two estimates to \eqref{eq:B(M,N)}, and from Theorem \ref{thm:exponentpair-constraints}, we obtain
\begin{align}
\cB(M, N)&\ll H^3N+X^\varepsilon (HLMN+(MN^2)^{\kappa+1}H^2L^{(\lambda+1)/2})\label{eq:B(M,N)-estimate},
\end{align}
provided that $\lambda\geqslant2\kappa,\lambda-\kappa\geqslant\frac{1}{2},$ and
\begin{align*}
(MN^2)^{1-2\kappa}\geqslant L^{\lambda-\frac{1}{2}}\end{align*}
if $(\kappa,\lambda)$ is of the type $A^k(\frac{1}{2},\frac{1}{2})$,
or that
\begin{align*}
(MN^2)^{2-2\lambda}\geqslant L^{\kappa}\end{align*}
if $(\kappa,\lambda)$ is of the type $BA^k(\frac{1}{2},\frac{1}{2})$.
Note that the term $HLMN$ comes from the {\it diagonal} terms with $h_2\sigma_2a_1=h_1\sigma_1a_2.$ 

The details of proving \eqref{eq:B(M,N)-estimate} have been omitted here, and attentive readers are referred to the Addendum attaching to the arXiv version of this paper: arXiv:1603.07060.

\subsection{Concluding Theorem \ref{thm:Brun-Titchmarsh}}
In order to conclude Proposition \ref{prop:Q_l(X)bound-average}, and thus Theorem \ref{thm:Brun-Titchmarsh}, 
it suffices to check, in view of \eqref{eq:B(M,N)bound}, that
\begin{align*}
H^3N+HLMN+(MN^2)^{\kappa+1}H^2L^{(\lambda+1)/2}\ll (ML)^{-1}(DL)^2X^{-\varepsilon'}\end{align*}
with the restriction $(MN^2)^{1-2\kappa}\geqslant L^{\lambda-\frac{1}{2}}$ or $(MN^2)^{2-2\lambda}\geqslant L^{\kappa}$ according to the type of $(\kappa,\lambda)$ satisfying $\lambda\geqslant2\kappa,\lambda-\kappa\geqslant\frac{1}{2}$.
The remaining task is to find the maximum of $D=MN$ when $M,N$ satisfy the above inequalities with some exponent pair $(\kappa,\lambda)\neq(\frac{1}{2},\frac{1}{2})$.
Put
\[M=X^\alpha,
\qquad
N=X^{\beta},
\qquad
D=X^{\gamma}=X^{\alpha+\beta},
\qquad
L=X^\theta.\]
We are going to maximize $\gamma=\alpha+\beta$ subject to the simultaneous restrictions
\begin{align*}
\begin{cases}
\alpha>0,\beta>0,\\
\alpha+\beta+\theta<\tfrac{3}{2},\\
\alpha+\theta<1,\\
(\alpha+2\beta)(\kappa+1)+\tfrac{1}{2}\theta(\lambda+3)+\alpha<2,\\
\lambda\geqslant2\kappa,\lambda-\kappa\geqslant\frac{1}{2}
\end{cases}
\end{align*}
with an additional constraint
\[(1-2\kappa)(\alpha+2\beta)\geqslant(\lambda-\tfrac{1}{2})\theta~~~~\text{~~~~or~~~~}~~~~2(1-\lambda)(\alpha+2\beta)\geqslant\kappa\theta,\]
depending that $(\kappa,\lambda)$ comes from $A^k$- or $BA^k$-process.

Choosing different exponent pairs, we may conclude different 
maximum of $\gamma$ when $\theta$ is in different ranges. Keeping in mind the amiability of the relevant trace function, we list the choices of exponent pairs as the following table .

\begin{table}[htbp]
\renewcommand{\arraystretch}{1.7}
\renewcommand{\tabcolsep}{2.2mm}
\centering
\begin{tabular}{ | c | c | c | c |c |}
\hline
$(\kappa,\lambda)$ 
& 
$(\frac{1}{6}, \frac{2}{3}) = A(\frac{1}{2}, \frac{1}{2})$ 
& 
$(\frac{1}{14},\frac{11}{14}) = A^2(\frac{1}{2}, \frac{1}{2})$ 
& 
$(\frac{1}{30},\frac{26}{30}) = A^3(\frac{1}{2}, \frac{1}{2})$ 
\\ 
\hline
maximum of $\gamma$
& 
$\frac{19-18\theta}{14}$ 
& 
$\frac{86-83\theta}{60}$ 
& 
$\frac{91-89\theta}{62}$ 
\\ 
\hline
range of $\theta$ 
&
$[\frac{1}{2},\frac{16}{17}[$
&
$[\frac{1}{2}, \frac{8}{9}[$
&
$[\frac{1}{2},\frac{112}{131}[$
\\ 
\hline
\end{tabular}
\\~\\
\vskip 2mm
\caption{Choices of $(\theta, \gamma)$}
\end{table}
Proposition \ref{prop:Q_l(X)bound-average} follows by optimizing the maximum of $\gamma$, and this completes the proof of Theorem \ref{thm:Brun-Titchmarsh}.

\smallskip

\section{Proof of Theorems \ref{thm:tauAPs} and \ref{thm:subconvexity}: Divisor functions in arithmetic progressions and subconvexity of Dirichlet $L$-functions}
\label{sec:tauAPs-subconvexity}

We now give the sketch of the proof of Theorem \ref{thm:tauAPs}. Following the arguments of Irving, it suffices to estimate the exponential sum
\begin{align*}
S := \sum_{n\sim N} \mathrm{e}\Big(\frac{h\overline{n}}{q}\Big)
\end{align*}
when $N\asymp \sqrt{X}$ and $(h,q)=1.$ Note that $n\mapsto \ue(h\overline{n}/q)$ is a compositely $\infty$-amiable trace function mod $q$, and so is its Fourier transform (as a normalized Kloosterman sum). 

Let $p$ be a large prime, and $\psi$ a non-trivial additive character on $\bF_p.$ It is known that $x\mapsto\psi(\overline{x})$ is of rank 1, geometrically irreducible and $\infty$-amiable. In fact, the underlying sheaf is universally amiable in the sense of Definition \ref{def:universallyamiable}, which allows us to apply any arithmetic exponent pair $(\kappa,\lambda,\nu)$ to the above average $S$, getting
\[S\ll N^{\varepsilon}(q/N)^\kappa N^\lambda\ll q^{\kappa}X^{(\lambda-\kappa)/2+\varepsilon}.\]
This is dominated by $X/\varphi(q)$ as long as
\[\theta<\frac{2}{3}+\frac{2-\kappa-3\lambda}{6(\kappa+1)}\cdot\]
Using the algorithm for exponent pairs (see \cite[Section 5]{GK91}), we would like to choose
\begin{align*}
(\kappa,\lambda)&=BA^3BA^2BABABA^2(\tfrac{1}{2},\tfrac{1}{2})=(\tfrac{591}{1535}, \tfrac{808}{1535}),\end{align*}
getting
\[\frac{2-\kappa-3\lambda}{6(\kappa+1)}=\frac{55}{12756}\approx\frac{1}{231.92},\]
which completes the proof of Theorem \ref{thm:tauAPs}.

\bigskip

We now turn to the proof of Theorem \ref{thm:subconvexity}.
Using the approximate functional equation of $L(\tfrac{1}{2}, \chi)$, it suffices to consider the incomplete character sum
\[\sum_{M<n\leqslant M+N}\chi(n).\]
Given a large prime $p$, the Kummer sheaf $\cL_\chi$ is of course universally amiable for each non-trivial character $\chi$ on $\bF_p^\times.$
In this case, for any arithmetic exponent pair $(\kappa,\lambda,\nu)$, we have
\[\sum_{M<n\leqslant M+N}\chi(n)\ll q^{\kappa}N^{\lambda-\kappa+O(\eta)},\]
and
\[
L(\tfrac{1}{2}, \chi)\ll  q^{\frac{\lambda+\kappa}{2}-\frac{1}{4}+O(\eta)}.
\]

Taking
\begin{align*}
(\kappa,\lambda)=ABA^3BA^2BABABA^2BABABA^2BA^2\cdots(\tfrac{1}{2},\tfrac{1}{2}),
\end{align*}
and according to Rankin (see also \cite[Section 5.4]{GK91}), the minimal value for $\kappa+\lambda$ should be $0.829021\ldots$,
which yields the expected exponent in Theorem \ref{thm:subconvexity}.

\smallskip

\appendix

\section{Estimates for complete exponential sums}

This appendix is devoted to estimate the complete exponential sum
\begin{align*}
\Sigma(\lambda,c)
:= \sum_{a\bmod{c}} \mathrm{e}\Big(\frac{\lambda(a)}{c}\Big),
\end{align*}
where $c$ is a fixed positive integer and $\lambda=\lambda_1/\lambda_2$ with $\lambda_1,\lambda_2\in\bZ[X]$ and $\lambda_1,\lambda_2$ being coprime in $\bZ[X]$.
The values of $a$ such that $(\lambda_2(a),c)\neq1$ are excluded from summation. 
We define the degree of $\lambda$ by
\[d=d(\lambda)=\deg(\lambda_1)+\deg(\lambda_2).\]
If 
\[\lambda_1(x)=\sum_{0\leqslant j\leqslant d_1}r_jx^j\in\bZ[x],\ \ \ \lambda_2(x)=\sum_{0\leqslant j\leqslant d_2}t_jx^j\in\bZ[x],\]
we then adopt the convention that 
\begin{align}
(\lambda,c_0)_*&=(r_0,r_1,r_2,\dots,r_{d_1},c_0),
\label{eq:gcdrationalfunctions1}\\
(\lambda,c_0)&=(r_1,r_2,\dots,r_{d_1},t_1,t_2,\dots,t_{d_2},c_0),\label{eq:gcdrationalfunctions2}\\ (\lambda',c_0)&=(\lambda_1'\lambda_2-\lambda_1\lambda_2',c_0)
\label{eq:gcdrationalfunctions3}
\end{align}
for all $c_0\mid c.$

There are many known estimates for complete exponential sums studied extensively by Vinogradov, Hua, Vaughan, Cochrane--Zheng, {\it et al}. We here present an alternative estimate that suits well for our applications to Theorem \ref{thm:Brun-Titchmarsh}.

By the Chinese remainder theorem, we have
\begin{align}\label{eq:CRT}
\Sigma(\lambda,c)=\prod_{p^\beta\|c}\Sigma(\overline{c/p^\beta}\cdot \lambda,p^\beta),\end{align}
where $\overline{c/p^\beta}$ denotes the multiplicative inverse of ${c/p^\beta}\bmod{p^\beta}$. Therefore, the evaluation of $\Sigma(\lambda,c)$ can be reduced to the case of prime power moduli.

The case $c=p$ can be guaranteed by Weil's proof on Riemann hypothesis for curves over finite fields. The estimate of a general form can be found, for instance, in \cite[Theorem 5]{Bo66}.
\begin{lemma}\label{lm:Weil}
Let $d=d(\lambda)\geqslant1.$ Then we have
\begin{align*}|\Sigma(\lambda,p)|\leqslant2dp^{\frac{1}{2}}(\lambda,p)^{\frac{1}{2}}(\lambda,p)_*^{\frac{1}{2}},
\end{align*}
where $(\lambda,p)$ and $(\lambda,p)_*$ are given following the convention in $\eqref{eq:gcdrationalfunctions1}, \eqref{eq:gcdrationalfunctions2}.$
\end{lemma}

In fact, one can make use of a better choice $2dp^{\frac{1}{2}}\{(\lambda,p)+(\lambda,p)_*\}^{\frac{1}{2}}$
in the upper bound in Lemma \ref{lm:Weil}. The choice $2dp^{\frac{1}{2}}(\lambda,p)^{\frac{1}{2}}(\lambda,p)_*^{\frac{1}{2}}$ makes Lemma \ref{lm:Weil} worser than the trivial bound if $\lambda_1$
is zero mod $p$; however, this choice makes the upper bound (up to some scalar) is multiplicative in $p$; which is crucial in the statement of Theorem \ref{thm:S(lambda,c)bound}.

The condition $(\lambda,p)=(\lambda,p)_*=1$ is equivalent to the statement that $\lambda\bmod p$ is a not constant function, in which case
 the above bound presents the {\it square-root cancellation} among the exponentials, which is best possible in general.
The case $c=p^\beta$ with $\beta\geqslant2$ becomes considerably easier because elementary evaluations are usually sufficient.
The following lemma evaluates $\Sigma(\lambda,p^\beta)$ for $\beta\geqslant2$, from which one can also obtain square-root cancellations up to some extra factor, which can be controlled effectively on average, as we will see later. The detailed proof can be found in \cite[Section 12.3]{IK04}.
\begin{lemma}\label{lm:IKevaluation}
Let $\alpha\geqslant1.$ Then we have
\begin{align*}
\Sigma(\lambda,p^{2\alpha})
= p^\alpha \sum_{\substack{y\bmod {p^\alpha}\\ \lambda'(y)\equiv0\!\bmod{p^\alpha}}}
\mathrm{e}\Big(\frac{\lambda(y)}{p^{2\alpha}}\Big)
\end{align*}
\end{lemma}

Using the evaluations as above, we would like to derive a precise estimate for $\Sigma(\lambda,c)$ that is applicable in many applications.

Recalling the expression \eqref{eq:CRT}, we may write
\begin{align*}
\Sigma(\lambda,c)
= \prod_{\substack{p^\beta\|c\\ \beta\leqslant2}}\Sigma(\overline{c/p^\beta}\cdot \lambda,p^\beta)\cdot\prod_{\substack{p^\beta\|c\\ \beta\geqslant3}}\Sigma(\overline{c/p^\beta}\cdot \lambda,p^\beta)
=: \Sigma_1\cdot\Sigma_2,
\quad(\text{say}).
\end{align*}
For $p\|c,$ by Lemma \ref{lm:Weil}, we have
\begin{align*}|\Sigma(\overline{c/p}\cdot \lambda,p)|\leqslant (2d)(\lambda,p)^{1/2}(\lambda,p)_*^{\frac{1}{2}}p^{\frac{1}{2}}.\end{align*}
If $p^2\|c$, Lemma \ref{lm:IKevaluation} implies 
\begin{align*}
|\Sigma(\overline{c/p^2}\cdot \lambda,p^2)|
& \leqslant p\,|\{y\bmod p:\lambda'(y)\equiv0\bmod p\}|
\\
& \leqslant \deg(\lambda')(\lambda',p)p.\end{align*}
Hence we conclude that
\begin{align*}
|\Sigma_1| \leqslant c_1^{\frac{1}{2}}(\lambda,c^\flat)^{\frac{1}{2}}(\lambda,c^\flat)_*^{\frac{1}{2}}(\lambda',c^\ddagger)(2d)^{\omega(c_1)}
\quad\text{with}\quad
c_1= \prod_{p^\beta\|c, \, \beta\leqslant2} p^\beta,\, c^\ddagger=\prod_{p^2\| c}p.
\end{align*}

On the other hand, a trivial estimate gives
\begin{align*}
|\Sigma_2|
\leqslant \prod_{p^\beta\|c, \, \beta\geqslant3} p^\beta = \Xi(c),
\end{align*}
where $\Xi(c)$ is defined by \eqref{eq:Xi}.

Collecting the above two estimates, we may conclude the following theorem.
\begin{theorem}\label{thm:S(lambda,c)bound}
Let $d=d(\lambda).$ For $c\geqslant1,$ we have
\begin{align*}
|\Sigma(\lambda,c)|
\leqslant c^\frac{1}{2}(\lambda,c^\flat)^{\frac{1}{2}}(\lambda,c^\flat)_*^{\frac{1}{2}}(\lambda',c^\ddagger)(2d)^{\omega(c)}\cdot\Xi(c)^{\frac{1}{2}},
\end{align*}
where $c^\ddagger$ is defined as $\eqref{eq:definition-ddagger}.$
\end{theorem}

\begin{remark}
We do not intend to seek the strongest estimate for $\Sigma(\lambda,c)$. 
Our interest here is to present a square-root cancellation up to some harmless factors. 
It is not difficult to check that $\Xi(c)$ is bounded by $\log c$ on average, 
which is acceptable particularly in our applications to the quadratic Brun--Titchmarsh theorem.
\end{remark}

\smallskip

\section{Universally amiable sheaves (by Will Sawin)}\label{appendix:uasheaves}

We begin with a brief review of local monodromy representations and slopes.

Any constructible sheaf on $\mathbf A^1_{\mathbf F_p}$ is lisse on some open subset $U$ of $\mathbf A^1_{\mathbf F_p}$. This lisse sheaf is equivalent to a representation of the \'{e}tale fundamental group $\pi_1^{ \textrm{et}} (U)$ of $U$, equivalently, a representation of $\operatorname{Gal}(\mathbf F_p(X))$ unramified at every closed point in $U$. We can restrict this representation to the inertia group of a place of $\mathbf F_p(X)$, i.e. a closed point of $\mathbf P^1_{\mathbf F_p} \setminus U$. This restriction is called the \emph{local monodromy representation} of $\mathcal F$ at that point.

The inertia group of $\operatorname{Gal}(\mathbf F_p(X))$ at any point has a canonical filtration, the upper numbering filtration, into normal subgroups $I^s$ indexed by nonnegative real numbers $s$. For an irreducible representation of the inertia group, the \emph{slope} is the infimum of all $s$ such that $I^s$ acts trivially on the representation. The slope is called the \emph{break} in older literature. It is known to always be a rational number. For a representation of the inertia group, not necessarily irreducible, we say its slopes are the set of all slopes of its Jordan-H\"older factors. Because $I^s$ is a pro-$p$ group for any $s>0$ and thus acts semisimply on any $\ell$-adic representation, it follows that a representation has all slopes $\leq 1$ if and only if it is invariant under $I^s$ for all $s>1$.

\begin{definition}[Universally amiable sheaf]\label{def:universallyamiable} 
An admissible sheaf $\mathcal F_p$ on $\bP^1_{\bF_p}$ is universally amiable if it is a geometrically isotypic Fourier sheaf $($in the sense of Definition $\ref{def:fouriersheaf})$ and its local monodromy at $\infty$ has all slopes $\leq 1$. 
\end{definition}

The purpose of this definition is that it is a simple, easy to check condition that ensures that the $q$-van der Corput method discussed in the main body of the paper can be used for the trace function with an arbitary sequence of $A$- and $B$-processes.

 We provide some examples and non-examples of univerally amiable sheaves, to explain how this definition includes many sheaves whose trace functions are of interest in analytic number theory but not certain sheaves where the $q$-van der Corput method obviously runs into trouble after a fixed sequence of $A$- and $B$-processes.
 
 While all these examples have a unique slope of its local monodromy representation at $\infty$, examples with multiple slopes at $\infty$ are also common whenever slightly more complicated trace functions appear.
 
 \begin{lemma}\label{lm:ua-Kummer-Example} Let $\chi: \bF_p^\times \to \bC^\times$ be a multiplicative character of order $d>1$. Let $f \in \mathbf F_p(X)$ be a rational function. Assume that $f$ is nonconstant and has no pole or zero of order divisible by $d$. Then $\mathcal L_{\chi(f)}$ is a universally amiable sheaf.\end{lemma}
  
\begin{proof} The facts that it is a middle extension sheaf, irreducible and thus isotypic, and pure of weight zero are all standard. The main subtlety is that a pole or zero of order $d$ would cause $\mathcal L_{\chi(f)}$ to fail to be middle extension at that point, but this could be repaired by removing the corresponding $d$'th power of a factor from $f$, which would not change the trace function away from a small number of points.
 
 Because the order $d$ of $\chi$ divides $p-1$ and thus is prime to $p$, the local monodromy representation of $\mathcal L_{\chi(f)}$ at $\infty$ is invariant under $I^s$ for all $s>0$, and so its (unique) slope at $\infty$ is $0$, which is certainly $\leq 1$. Every Artin-Schreier sheaf of a linear function has slope $1$ at infinity unless it is the constant sheaf, so to check that $\mathcal L_{\chi(f)}$ has no Artin-Schreier component is suffices to check it has no constant component, which follows from the fact that $f$ is not a constant function. (We use the fact that $f$ has no poles or zeroes of order a multiple of $d$ - otherwise we would have to assume $f$ is not a constant times a $d$'th power.)\end{proof}

 \begin{lemma}\label{lm:ua-Artin-Schreier-example} Let $\psi: \bF_p \to \bC^\times $ be an additive character. Let $f \in \mathbf F_p (X) $ be a rational function of degree $<p$. Assume that the degree of the numerator of $f$ is at most one more than the degree of the denominator of $f$ (i.e. that $f$ has a pole at $\infty$ of degree at most $1$) and that $f$ is not a polynomial of degree $1$. Then $\mathcal L_{\psi(f)}$ is a universally amiable sheaf.\end{lemma}
 
\begin{proof} The facts that $\mathcal L_{\psi(f)}$ is a middle extension sheaf, irreducible and thus isotypic, and pure of weight zero are standard. (A sufficient condition to be a middle extension sheaf is that $f$ has no pole of order a multiple of $p$, but this follows from $f$ having degree less than $p$.)
 
 The unique slope of $\mathcal L_{\psi(f)}$ at $\infty$ equals the order of the pole of $f$ at $\infty$, or $0$ if $f$ does not have a pole at $\infty$. (It would be less than this if the order of the pole were a multiple of $p$, but this again cannot happen by our assumption on degree.) Since we have assumed the pole order is $\leq 1$, the slope is $\leq 1$.
 
 Two Artin-Schreier sheaves are isomorphic if their rational functions $f_1,f_2$ differ by a term of the form $g^p-g$. If $f_1$ and $f_2$ have no poles of order at least $p$, this can only happen if $f_1$ and $f_2$ differ by a constant. So to check that $\mathcal L_{\psi(f)}$ is not geometrically isomorphic to $\mathcal L_{\psi(L)}$ for $L$ a polynomial of degree $1$, it suffices to check that $f$ is not a polynomial of degree $1$ plus a constant, which is again a polynomial of degree $1$. We have assumed this, verifying the last condition.
\end{proof}

\begin{lemma}\label{lm:ua-Kloosterman-example} Let $k \geq 2$ be a natural number, and let $\mathcal K\ell_k$ by the {\rm(}normalized{\rm)} hyper-Kloosterman sheaf constructed by Deligne and studied by Katz. Then $\mathcal K\ell_k$ is universally amiable. \end{lemma}

\begin{proof} The facts that it is a middle extension sheaf, geometrically irreducible, pure of weight $0$ (after normalization), and that its unique slope at $\infty$ is $1/k \leq 1$ all follow from work of Katz \cite{Ka88}. Because it is geometrically irreducible of rank $k >1$, it has no geometrically irreducible components of rank $1$, and in particular no Artin-Schreier components. \end{proof}

\begin{lemma}\label{lm:ua-Artin-Schreier-non-example} Let $f \in \mathbf F_p[X]$ be a polynomial. Then $\mathcal L_{\psi(f)}$ is not universally amiable.\end{lemma}

\begin{proof} We may first assume that the degree of $f$, if positive, is not divisible by $p$, as replacing a term $a X^{np}$ with $a X^n$ does not affect the sheaf $\mathcal L_{\psi(f)}$ or its trace function.

If the degree of $f$ is at most $1$, then $\mathcal L_{\psi(f)}$ is clearly isomorphic to a sheaf $\mathcal L_{\psi(L)}$ for $L=f$ a polynomial of degree $\leq 1$.

If the degree of $f$ is greater than $1$, then because the unique slope of $\mathcal L_{\psi(f)}$ at $\infty$ is the degree of $f$ (here we use the assumption that the degree is not divisible by $p$) and thus is not $\leq 1$. 
\end{proof}

\begin{lemma}\label{ua-FT-Artin-Schreier-non-example} Let $f \in \mathbf F_p[X]$ be a polynomial of degree $d$. Assume that $d \geq 2$ and $d$ is prime to $p$. Then $\operatorname{FT}_\psi ( \mathcal L_{\psi(f)})$ is not universally amiable.\end{lemma}

\begin{proof} Because $d$ is prime to $p$, the unique slope of the local monodromy representation of $\mathcal L_{\psi(f)}$ at $\infty$ is $d$. It then follows from Laumon's local Fourier transform theory \cite[Theorem 7.4.1(1)]{Ka90} that the local monodromy representation of $\operatorname{FT}_\psi ( \mathcal L_{\psi(f)})$ at $\infty$ has slope $\frac{d}{d-1}>1$, and thus  $\operatorname{FT}_\psi ( \mathcal L_{\psi(f)})$  is not universally amiable. \end{proof}

We now prove two lemmas that show the property of being universally amiable is stable under the sheaf-theoretic analogues of the $A$- and $B$-processes in the $q$-van der Corput method.

\begin{lemma}\label{lm:ua-A-process} Suppose $\mathcal F$ is a universally amiable sheaf on $
\bP^1_{\mathbf  F_p}$ with $\mathfrak c (\mathcal F) \leq p$ for some $p>2$. Then for each $a \in \mathbf F_p^\times$, there exist a subset $A\subset\bF_p$ and a sequence of universally amiable sheaves $\{\cF_j\}_{1\leqslant j\leqslant J}$ with $1\leqslant J\leqslant\mathfrak c(\mathcal F)^2,$ such that
the trace function of $[+a]^* \mathcal F \otimes \widecheck{\cF}$ can be expressed as
\begin{align*}
\sum_{1\leqslant j\leqslant J}K_j+\sum_{a\in A}\delta_a,
\end{align*}
where $K_j$ is the trace function of $\cF_j$ with $\fc(\cF_j)\leqslant \mathfrak c(\mathcal F)^4$ for each $1\leqslant j\leqslant J$, and $\sum_{a\in A}(1+\|\delta_a\|_\infty)\leqslant 2\fc(\cF)^3.$ 
\end{lemma}

\begin{proof} The tensor product is a middle extension away from points which are singular points of both $\cF$ and $\widecheck{\cF}$. The number of these is at most the number of singular points of $\cF$, which is at most $\mathfrak c(\cF)$. Replacing $[+a]^* \mathcal F \otimes \widecheck{\cF}$ by a middle extension sheaf changes the trace function by at most the rank of $[+a]^* \mathcal F \otimes \widecheck{\cF}$, which is the square of the rank of $\mathcal F$ and thus is at most the square of the conductor of $\mathcal F$.

Having done this, we can write the trace function as a sum of geometrically isotypic components and components with zero trace function (see \cite[Proposition 8.3]{FKM15}). Ignoring the components with zero trace function, the number of components is at most the rank of $[+a]^* \mathcal F \otimes \widecheck{\cF}$, which is again at most the square of the the conductor, and each component has conductor at most $5 \mathfrak c (\mathcal F)^4 $ by \cite[Proposition 8.2(3)]{FKM15}.

The remaining conditions to check are the condition on the local monodromy at $\infty$ and the condition on the geometrically irreducible components.

Because the local monodromy representation of $\cF$ at $\infty$ is invariant under $I^s$ for every $s>1$, the same is true of its dual representation, which is the local monodromy of $\widecheck{\cF}$, and its pullback by an automorphism of the local field, which is the local monodromy of $[+a]^* \cF$, because the higher numbering filtration is defined intrinsically and so is invariant under automorphisms. Because both tensor factors are invariant under $I^s$ for all $s>1$, the tensor product is as well, and so it has all slopes $\leq 1$.

Lastly, to check that (the middle extension analogue of) $[+a]^* \mathcal F \otimes \widecheck{\cF}$ contains no geometrically isotypic component isomorphic to $\mathcal L_{\psi(L)}$ for linear $L$, it suffices by \cite[Theorem 6.15]{Po14} to show that $\mathcal F$ has conductor $\leq p$ (which we assumed) and that $\mathcal F$ has no geometrically isotypic component isomorphic to $\mathcal L_{\psi(Q)} $ for a polynomial $Q$ of degree $\leq 2$. If $Q$ has degree $\leq 1$, this follows from our assumption that $\mathcal F$ has no such component. If $Q$ has degree exactly $2$, this follows from our assumption that the local monodromy representation of $\mathcal F$ at $\infty$ has all slopes $ \leq 1$, so all its components have all slopes $\leq 1$, but the local monodromy of $\mathcal L_{\psi(Q)}$ at $\infty$ has unique slope $2$ and so cannot be a component of the local monodromy of $\mathcal F$. \end{proof}

\begin{lemma}\label{lm:ua-B-process} Suppose that $\cF$ is a universally amiable sheaf on $\bP^1_{\mathbf F_p}$. Then $\operatorname{FT}_\psi(\cF)$ is universally amiable. \end{lemma}

\begin{proof} Because $\mathcal F$ is a middle extension sheaf and has no geometrically irreducible component isomorphic to an Artin-Schreier sheaf of a linear polynomial, it is a Fourier sheaf in the sense of \cite[(7.3.5)]{Ka90}. Furthermore, it is pure of weight $0$. Thus by \cite[Theorem 7.3.8(2,5)]{Ka90} its Fourier transform is Fourier and pure of weight $0$ (where we put a half-Tate twist in the definition of the Fourier transform that Katz does not). Because it is Fourier, it is a middle extension sheaf, and has no geometrically irreducible component isomorphic to an Artin-Schreier sheaf of a linear function.

Furthermore, because $\cF$ is geometrically isotypic, and the Fourier transform sends its identical geometrically irreducible components to identical geometrically irreducible components \cite[Theorem 7.3.8(3)]{Ka90}, its Fourier transform is also geometrically isotypic.

It remains to prove that the local monodromy representation of $\operatorname{FT}_\psi(\cF)$ at $\infty$ has all slopes $\leq 1$.

Laumon proved that the local monodromy representation of $\operatorname{FT}_\psi(\cF)$ at $\infty$ is a sum of contributions arising from the local monodromy representations fo $\cF$ at various points of $\mathbf P^1_{\mathbf F_p}$. Moreover, the contribution from the local monodromy representation at $0$ has all slopes $<1$, the contribution from any point other than $0$ or $\infty$ has all slopes $1$, and the contribution from $\infty$ has all slopes $>1$, but is nonvanishing if and only if the local mondromy representation of $\cF$ at $\infty$ has some slope $>1$. This is explained in \cite[Theorem 7.4.1]{Ka90}. Thus, because the local monodromy represntation of $\cF$ at $\infty$ has all slopes $\leq 1$ at $\infty$, the local monodromy representation of $\operatorname{FT}_\psi(\cF)$ at $\infty$ has all slopes $\leq 1$ as well. 
\end{proof}

\bibliographystyle{plain}

\bigskip

\end{document}